\documentclass[psamsfonts]{amsart}

\usepackage{amssymb,amsfonts}
\usepackage[all,arc]{xy}
\usepackage{enumerate}
\usepackage{mathrsfs}
\usepackage{amsmath}
\usepackage{bbm}
\usepackage{tikz}
\usepackage{pgfplots}
\usepackage{hyperref}
\usepackage[alphabetic]{amsrefs}
\usepackage{mathtools}

\DeclarePairedDelimiter{\floor}{\lfloor}{\rfloor}

\newtheorem{thm}{Theorem}[section]
\newtheorem{cor}[thm]{Corollary}
\newtheorem{prop}[thm]{Proposition}
\newtheorem{lem}[thm]{Lemma}
\newtheorem{conj}[thm]{Conjecture}

\theoremstyle{definition}
\newtheorem{defn}[thm]{Definition}

\newtheorem{claim}[thm]{Claim}

\theoremstyle{remark}
\newtheorem{rem}[thm]{Remark}

\newcommand{\R}{\mathbb{R}}
\newcommand{\F}{\Phi}

\newcommand{\e}{\varepsilon}
\newcommand{\C}{\mathbb{C}}
\newcommand{\Z}{\mathbb{Z}}
\newcommand{\z}{\mathcal{Z}}
\newcommand{\N}{\mathbb{N}}
\renewcommand{\P}{\mathcal{P}}

\renewcommand{\|}{|||}
\newcommand{\psix}{\psi_x}
\newcommand{\Psix}{\Psi_x}
\newcommand{\funs}{\mathcal{F}}
\newcommand{\logd}{ \text{d}^{\log} }

\newcommand{\K}{\mathcal{K}}
\newcommand{\E}{\mathbb{E}}

\newcommand{\ch}{\mathbbm{1}}
\newcommand{\vertiii}[1]{{\left\vert\kern-0.25ex\left\vert\kern-0.25ex\left\vert #1 
		\right\vert\kern-0.25ex\right\vert\kern-0.25ex\right\vert}}

\renewcommand{\Re}{\text{Re} }


\begin{document}
	\title{Sarnak's conjecture for sequences of almost quadratic word growth}
	\author{Redmond McNamara}
	\maketitle
	\begin{abstract}
		We prove the logarithmic Sarnak conjecture for sequences of subquadratic word growth. In particular, we show that the Liouville function has at least quadratically many sign patterns. We deduce the main theorem from a variant which bounds the correlations between multiplicative functions and sequences with subquadratically many words which occur with positive logarithmic density. This allows us to actually prove that our multiplicative functions do not locally correlate with sequences of subquadratic word growth. We also prove a conditional result which shows that if the $\kappa-1$-Fourier uniformity conjecture holds then the Liouville function does not correlate with sequences with $O(n^{t-\e})$ many words of length $n$ where $t = \kappa(\kappa+1)/2$. We prove a variant of the $1$-Fourier uniformity conjecture where the frequencies are restricted to any set of box dimension $< 1$.
	\end{abstract}
	\section{Introduction}

	The prime number theorem states that
	\[
	\lim_{N \rightarrow \infty} \E_{n \leq N} \Lambda(n) = 1,
	\]
	where $\Lambda(n) = \log p$ if $n$ is a power of a prime $p$ and $0$ otherwise is the von Mangoldt function. (We refer the reader to Section \ref{background} for an explanation of the $\E$ notation). This is equivalent to the estimate
	\[
	\lim_{N\rightarrow \infty} \E_{n \leq N} \lambda(n) = 0,
	\]
	where $\lambda(n) = (-1)^{\text{$\#$ of prime factors of n}}$ is the Liouville function.
	Dirichlet's theorem on prime numbers in arithmetic progressions morally follows from the estimate
	\[
	\lim_{N\rightarrow \infty} \E_{n \leq N} \ch_{n \equiv r \mod d} \lambda(n) = 0,
	\]
	for any $d$ and $r$. Taking linear combinations, we find that for any periodic function $f$,
	\[
	\lim_{N\rightarrow \infty} \E_{n \leq N} f(n) \lambda(n) = 0.
	\]
	Equivalently, for any function $F \colon S^1 \rightarrow \C$ and any rational angle $\alpha$,
	\[
	\lim_{N\rightarrow \infty} \E_{n \leq N} F(\alpha^n) \lambda(n) = 0.
	\]
	The analogous estimate when $\alpha$ is irrational and $F$ is a continuous function was proved by Vinogradov and was a key ingredient in his proof that any sufficiently large odd number is the sum of three primes. Green and Tao proved that 
\[
	\lim_{N\rightarrow \infty} \E_{n \leq N} F(g^n \Gamma) \lambda(n) = 0.
\]
	where $G$ is a nilpotent Lie group, $g$ is an element of $G$, $\Gamma$ is a cocompact lattice and $F$ is a continuous function $F \colon G / \Gamma \rightarrow \C$. A version of this statement was a key ingredient in their proof with Tamar Ziegler that counts the solutions to almost any system of linear equations over the primes. This motivates the following conjecture, due to Sarnak:
	\begin{conj}[Sarnak, see \cite{MR3014544} ]
		For any topological dynamical system $(X, T)$ with zero entropy, any continuous function $F \colon X \rightarrow \C$ and any point $x$ in $X$,
		\[
		\lim_{N\rightarrow \infty} \E_{n \leq N} F(T^n x) \lambda(n) = 0.
		\]
	\end{conj}
	Tao introduced the following variant,
	\begin{conj}[Logarithmically Averaged Sarnak Conjecture]
		For any topological dynamical system $(X, T)$ with zero entropy, any continuous function $F \colon X \rightarrow \C$ and any point $x$ in $X$,
		\[
		\lim_{N\rightarrow \infty} \E^{\log}_{n \leq N} F(T^n x) \lambda(n) = 0.
		\]
	\end{conj}
	Many instances of Sarnak's conjecture have been proven. We give a few examples but stress that this is an incomplete list: \cite{MR3043150} \cite{MR2986954} \cite{MR3095150} \cite{MR3459905} \cite{MR3415586} \cite{MR3622068} \cite{MR3819999} \cite{MR3631324} \cite{MR3347317} \cite{MR3263939} \cite{MR3724218} \cite{MR3859364} \cite{MR3702497} \cite{MR3660308}.
	
\begin{defn}
	A word $\epsilon$ of length $k$ is an element of $\C^k$.
	Let $k$ be a natural number, let $\epsilon \in \C^k$ and let $b \colon \N \rightarrow \C$. We say that $\epsilon$ occurs as a word of $b$ if there exists a natural number $n$ such that $b(n+h) = \epsilon_h$ for all $h \leq k$. We say that $\epsilon$ occurs with (upper) logarithmic density $\delta \in \R$ if 
	\[
	\limsup_{N \rightarrow \infty} \E^{\log}_{n \leq N} \ch_{\epsilon_h = b(n+h) \text{ for all $h \leq k$ }}  = \delta.
	\]
	In this paper, when we refer to $\log$-density we mean upper logarithmic density.
	A word $\epsilon$ whose entries are all $\pm 1$ is called a sign pattern. We say that $b$ has subquadratic word growth if $b$ takes finitely many possible values and the number of words of length $k$ that occur with positive upper logarithmic density is $o(k^2)$.
\end{defn}
	
	Then a particular case of Sarnak's conjecture predicts that for any bounded sequence $b \colon \N\rightarrow \C$ with subexponential word growth that
	\[
	\lim_{N\rightarrow \infty} \E^{\log}_{n \leq N} b(n) \lambda(n) = 0.
	\]
	Because $\lambda$ correlates with itself, this in particular implies that the number of sign patterns of $\lambda$ of length $k$ is exponential in $k$. \cite{FrantzikinakisHost} proved the special case where $b$ has linear word growth.
	In this paper, we prove the following special case:
	\begin{thm}\label{subquadforl}
		Let $b$ be a bounded sequence with subquadratic word growth. Then
		\[
		\lim_{N\rightarrow \infty} \E^{\log}_{n \leq N} b(n) \lambda(n) = 0.
		\]
	\end{thm}
	Previously \cite{Hildebrand} showed that all 8 sign patterns of length 3 occur infinitely often. \cite{MRT2} showed all 8 sign patterns of length 3 occur with positive density. \cite{TJ} proved that all 16 sign patterns of length 4 occur with positive density using an argument communicated to them by Matom\"aki and Sawin. \cite{TJ} also showed the number of sign patterns of length $k$ is at least $2k + 8$ for $k \geq 4$. \cite{FrantzikinakisHost} showed that the number of sign patterns is super linear.
	In particular, Theorem \ref{subquadforl} implies that $\lambda$ does not have subquadratically many sign patterns. We actually prove something slightly stronger.
	\begin{thm}\label{signpatternsofl}
		There is a constant $\delta > 0$ such that $\lambda$ has at least $\delta k^2$ many sign patterns of length $k$. 
	\end{thm}
	\cite{TaoEquiv} showed that the $\log$ Sarnak conjecture is equivalent to the following Fourier uniformity conjecture for every natural number $t$.
	\begin{conj}[$t$-Fourier uniformity]
		Let $G$ be a nilpotent Lie group of step $t$, $\Gamma$ a cocompact lattice and $F \colon G / \Gamma \rightarrow \C$ a continuous function. Then
		\[
		\lim_{H \rightarrow \infty} \lim_{N \rightarrow \infty} \E^{\log}_{n \leq N} \sup_{g \in G} | \E_{h \leq H} \lambda(n + h) F(g^h \Gamma) | = 0.
		\]
	\end{conj}
	\cite{TaoEquiv} also showed that this is equivalent to the $\log$-Chowla conjecture for every $t$.
	\begin{conj}[Logarithmic Chowla Conjecture]
		For every natural number $t$ and every distinct natural numbers $h_1, \ldots, h_t$, we have
		\[
		\lim_{N \rightarrow \infty} \E^{\log}_{n \leq N} \lambda(n + h_1) \cdots \lambda(n+ h_t) = 0.
		\]
	\end{conj}
	A function $a \colon \N \rightarrow \C$ is said to be unpretentious, nonpretentious or strongly aperiodic if there exists a function $\phi$ from $\N$ to $\N$ such that, for all natural numbers $A$, for all Dirichlet characters $\chi$ of period at most $A$ we have, for all natural numbers $N$ sufficiently large and for all real numbers $|t| \leq A N$ we have
	\[
	\sum_{p \leq N} \frac{1 - \text{Re} (a(p) \chi(p) p^{-it} )}{p} \geq \phi(A),
	\]
	and $\phi(A) \rightarrow \infty$ as $A \rightarrow \infty$.
	The main goal of this paper is to prove the following theorems.
	\begin{thm}\label{mainthm}
		Let $a \colon \N \rightarrow S^1$ be an unpretentious completely multiplicative function taking values in the unit circle. Let $b \colon \N \rightarrow \C$ be a finite-valued $1$-bounded function. Suppose further that for any $\delta > 0$ there are infinitely many $k$ such that the number of words of $b$ of length $k$ that occur with positive upper logarithmic density is at most $\delta k^2$. Then
		\[
		\lim_{N \rightarrow \infty} \big| \E_{n \leq N}^{\log} a(n) b(n) \big| = 0.
		\]
	\end{thm}
	We also obtain a conditional version of this result.
	\begin{thm}\label{condmainthm}
		Let $\kappa$ be a natural number. Set $t = {\kappa + 1 \choose 2}$.
		Let $a \colon \N \rightarrow S^1$ be an unpretentious completely multiplicative function taking values in the unit circle so that the local $\kappa-1$-Fourier uniformity conjecture holds for $a$. Let $b \colon \N \rightarrow \C$ be a finite-valued $1$-bounded function. Suppose further that for some $\epsilon > 0$ there are infinitely many $k$ such that the number of words of $b$ of length $k$ that occur with positive upper logarithmic density is at most $k^{t - \e}$. Then
		\[
		\lim_{N \rightarrow \infty} \big| \E_{n \leq N}^{\log} a(n) b(n) \big| = 0.
		\]
	\end{thm}
	We note that this result matches the numerology in \cite{Sawin} and may be almost the best possible result one can obtain with purely dynamical methods. We also note that even the $1-$Fourier uniformity conjecture is still unknown and so this theorem currently has no unconditional content.
	We also obtain a version of the theorem where $b$ need not take only finitely many values and we only have information about the number of ``approximate" words.
	\begin{defn}
		We say a sequence $b$ has at most $h$ words of length $k$ up to $\e$ rounding if there exists a set $\Sigma$ of words of length $k$ such that for all $n \in \N$ there is an $\epsilon$ in $\Sigma$ such that $| b(n + j) - \epsilon_j | \leq \e$ for all $j \leq k$ and the cardinality of $\Sigma$ is at most $h$. We say $b$ has at most $h$ words of length $k$ that occur with positive logarithmic density up to $\e$ rounding if we only require $| b(n + j) - \epsilon_j | \leq \e$ for a set of $n$ of lower logarithmic density $1$.
	\end{defn} 
	\begin{thm}\label{approxmainthm}
		Let $c > 0$ and $\e > 0$. Then if $\e$ is sufficiently small depending on $c$ then the following holds:
		Let $a \colon \N \rightarrow S^1$ be an unpretentious completely multiplicative function taking values in the unit circle. Let $b \colon \N \rightarrow \C$ be a $1$-bounded function with entropy zero. Suppose further that for every $\delta > 0$ there are infinitely many $k$ such that the number of words of $b$ of length $k$ that occur with positive logarithmic density up to $\e$ rounding is at most $\delta k$. Then
		\[
		\limsup_{N \rightarrow \infty} \big| \E_{n \leq N}^{\log} a(n) b(n) \big| \leq c.
		\]
		In fact, this works for any $\e$ satisfying $c^2 > 2\e$.
	\end{thm}
	We list a few new applications of this theorem.
\\ {\textit{Proof of Theorem \ref{signpatternsofl}.}}
		Apply Theorem \ref{mainthm} to $a = b = \lambda$.
\hfill$\square$ \\
	\begin{thm}\label{localcorrelations}
		If $S$ is a finite set of sequences of subquadratic word growth and $a$ is an unpretentious completely multiplicative function taking values in the unit circle then
		\[
		\lim_{H \rightarrow \infty} \lim_{N \rightarrow \infty} \E^{\log}_{n \leq N} \sup_{\phi \in S} | \E_{h \leq H} a(n + h) \phi(h) | = 0.
		\]
	\end{thm}
	\begin{rem}
		We remark that since the set $S$ is finite, it is enough to show that any one function does not locally correlate with $a$. However, we also remark that it is generally harder to show that $a$ does not locally correlate with $b$ than it is to show that $a$ does not correlate with $b$. For Theorem \ref{localcorrelations}, we need to use that Theorem \ref{mainthm} allows us to handle the case where $b$ may have many words which occur with $0$ $\log$-density but still only subquadratically many words which occur with positive $\log$-density. Theorem \ref{localcorrelations} in the linear word growth case seems to follow implicitly from \cite{alex2019mbius}. 
	\end{rem}
	\begin{proof}
		For convenience, we will assume that $0$ is in $S$.
		Let $\e > 0$. We aim to show that
		\[
		\limsup_{H \rightarrow \infty} \limsup_{N \rightarrow \infty} \E^{\log}_{n \leq N} \sup_{\phi \in S} | \E_{h \leq H} a(n + h) \phi(h) | = O(\e).
		\]
		Suppose not. We will now use an argument of \cite{TaoEquiv} (see Section 5 of that paper) to show that $a$ must be correlate with a ``ticker tape" function.
		We define $S_\e$ to be the set of sequences of the form $\phi'(n) = e(\alpha) \phi(n)$ where $\phi$ is an element of $S$ and $\alpha$ is a rational number with denominator $O(\e)$. By the pigeonhole principle, for any $\phi$ in $S$ and any natural numbers $H$ and $n$ in $\N$ there exists $\alpha$ a rational number with denominator $O(\e)$ such that 
		\[
		\text{Re}\Big( |\E^{\log}_{h \leq H} \phi(h)a(n+h)| - \E^{\log}_{h \leq H} e(\alpha) \phi(h)a(n+h) \Big) = O(\e).
		\]
		Therefore, we may assume for the sake of contradiction that for some $\phi_{n, H}$ in $S_\e$
		\[
		\limsup_{H \rightarrow \infty} \limsup_{N \rightarrow \infty} \Re \Big( \E^{\log}_{n \leq N}  \E_{h \leq H} a(n + h) \phi_{n, H}(h) \Big) \gg \e.
		\]
		By a diagonalization argument, we may find a sequence $H_i$ and $N_i$ of natural numbers both tending to infinity and functions $\phi_{n, i} = \phi_{n, H_i}$ such that $N_{i+1} \gg N_i \gg H_i$ and
		\[
		\lim_{i \rightarrow \infty} \Re\Big( \E^{\log}_{n \leq N_i} \E_{h \leq H_i} a(n + h) \phi_{n,i}(h) \Big)  \gg \e.
		\]
		Since the functions $\phi$ in $S_\e$ and $a$ are bounded, for $i$ sufficiently large there exists a set $A_i$ of natural numbers of lower logarithmic density $\gg \e$ in the interval $[1, N_i]$ such that for $n$ in $A_i$,
		\[
		\Re \Big( \E_{h \leq H_i} a(n + h) \phi_{n,i}(h) \Big) \gg \e.
		\]
		By a greedy algorithm, we can select a subset $B_i$ of $A_i$ of upper logarithmic density at least $\frac{\e}{H_i}$ in $[1, N_i]$ that is at least $H_i$ separated (meaning distinct points of $B_i$ differ by at least $H_i$).
		Now define the ``ticker tape" function $\psi$ as follows:
		\[
		\psi(n+h) = \phi_{n,i}(h),
		\]
		for all $n$ in $B_i$ between $N_{i-1}$ and $N_{i}$ and $h \leq H_i$. If $m$ is not of the form $n + h$ for $n$ in $B_i$ between $N_{i-1}$ and $N_i$ and $h \leq H_i$ then we set $\psi(m) = 0$. Thus,
		\[
		\limsup_{N \rightarrow \infty} \Re \Big( \E^{\log}_{n \leq N} a(n) \psi(n) \Big) \gg \e^2.
		\]
		Now we aim to show that $\psi$ has subquadratically many words of length $k$ that occur with positive upper logarithmic density. Let $k$ be a natural number and let $\epsilon$ be a word of length $k$ which occurs in $\psi$ with positive upper logarithmic density. Consider the set $C$ of natural numbers $m$ such that $m$ is within $k$ of an element $n$ of $B_i$ or $B_i + H_i$ for some $i$. Then since elements of $B_i$ are at least $H_i$ separated, the upper logarithmic density of $C$ in $[N_{i-1}, N_{i}]$ is at most $\frac{2k}{H_i}$ which clearly tends to $0$ as $i$ tends to infinity. Since $N_i \gg N_{i-1}$, we may assume that the $\log$-density of $[1,N_{i-1}]$ in the interval $[1,N_i]$ is also $o(1)$. Thus, $C$ has $\log$-density $0$. Therefore, if $\epsilon$ occurs with positive $\log$-density then $(\psi(n+h))_{h=1}^k = \epsilon$ for a positive density set of $n$ not in $C$. Since $S_\e$ has only finitely many members, we get that there exists $\phi$ in $S_\e$ such that for a positive upper logarithmic density set of $n$, $\psi(n+h) = \phi(n+h) = \epsilon_h$ for all $h \leq k$. Thus, $\psi$ has subquadratic word growth and $a$ does not correlate with $\psi$ by Theorem \ref{mainthm}, which gives a contradiction.
	\end{proof}
	\begin{thm}\label{Cantor}
		Let $C$ be a subset of $[0,1]$ of upper box dimension $< 1$. Then if $a$ is an unpretentious completely multiplicative function taking values in the unit circle
		\[
		\lim_{H \rightarrow \infty} \lim_{N \rightarrow \infty} \E^{\log}_{n \leq N} \sup_{\alpha \in C} | \E_{h \leq H} a(n + h) e(h\alpha) | = 0.
		\]
	\end{thm}
	\begin{rem}
		In particular, this implies that if $C$ is the middle thirds Cantor set then
		\[
		\lim_{H \rightarrow \infty} \lim_{N \rightarrow \infty} \E^{\log}_{n \leq N} \sup_{\alpha \in C} | \E_{h \leq H} a(n + h) e(h\alpha) | = 0.
		\]
		Of course, the result also applies to a large family of other fractals.
		The author does not know of any results in the literature where this is established for any infinite set. He does not know of any proof for any set of positive box dimension which does not use Theorem \ref{approxmainthm}.
	\end{rem}
	\begin{proof}
		Suppose the upper box dimension of $C \subset S^1$ is $< d < 1$. 
		Let $\e > 0$. As in the proof of Theorem \ref{localcorrelations}, we assume that
		\[
		\limsup_{H \rightarrow \infty} \limsup_{N \rightarrow \infty} \E^{\log}_{n \leq N} \sup_{\alpha \in C} | \E_{h \leq H} a(n + h) e(h\alpha) | \gg \e,
		\]
		and derive a contradiction. As before, there is a ticker tape function $\psi \colon \N \rightarrow \C$ such that
		\[
		\limsup_{N \rightarrow \infty} \Re \Big( \E^{\log}_{n \leq N}  \E_{h \leq H} a(n + h) \psi(h) \Big) \gg \e^2,
		\]
		of the following form: there exists sequences of natural numbers $N_i$ and $H_i$ tending to infinity with $N_{i+1} \gg N_i \gg H_i$, a sequence of $H_i$-separated sets $B_i$, and $\psi(n+h) = e(\beta_n) e(\alpha_n h)$ for some rational $\beta_n$ of denominator at most $O(\e)$, some $\alpha_n$ in $C$ and for all $n$ in some set $B_i \cap [N_{i-1}, N_i]$ and $h \leq H_i$. We set $\psi(m) = 0$ for all natural numbers $m$ not of this form. As before, for any natural number $k$, the natural numbers $m$ that are within $k$ of a number $n$ in $B_i$ or $B_i + H_i$ has $\log$-density $0$.
		
		Let $k$ be a natural number sufficiently large depending on $C$ and $\e$. Let $\e' = \e^2$. Then because $C$ has upper box dimension $< d$ there exists a collection of at most $(\frac{k}{\e'})^d$ intervals of length $\frac{\e'}{k}$ covering $C$. If two points on the circle $\alpha$ and $\alpha'$ differ by at most $\frac{\e'}{k}$ then by the triangle inequality, for all $h \leq k$, we have that $| e(h \alpha) - e(h \alpha') | \leq \e'$. Therefore, the number of sign patterns of $\psi$ that occur with positive $\log$-density up to $\e'$ rounding is sublinear. In particular, for any $\delta > 0$, there are fewer than $\delta k$ many sign patterns that occur with positive $\log$-density up to $\e^2$ rounding.  By Theorem \ref{approxmainthm}, we get a contradiction.
	\end{proof}
	On the surface, this argument appears to be very close to the $t$-Fourier uniformity conjecture, which Tao introduced in \cite{TaoEquiv} and proved was equivalent to the $\log$-Chowla and $\log$-Sarnak conjectures. (For recent significant progress on the Fourier uniformity conjecture, see \cite{MRT3}). If you wanted to prove the Fourier uniformity conjecture in the case $d = 1$, namely that
	\[
	\lim_{H \rightarrow \infty} \limsup_{N \rightarrow \infty} \E^{\log}_{n \leq N} \sup_{\alpha \in \R} | \E_{h \leq H} a(n + h) e(h\alpha) | = 0,
	\]
	the ticker tape functions that you would need $\lambda$ to be orthogonal to have $\sim \e^{-1} k$ many sign patterns of length $k$ up to $\e$ rounding. Thus, one might hope that a simple argument could adjust the constants in Theorem \ref{approxmainthm} and thereby prove the Fourier uniformity conjecture. However, there is a major theoretical obstacle to further progress. \cite{FrantzikinakisHost} introduced the dynamical system $(S^1 \times S^1, dx, T, \mathcal{B})$ where $T(\alpha, \beta) = (\alpha, \alpha \beta)$. \cite{Sawin} showed that this dynamical system with some additional structure is a dynamical model for the Liouville function (a notion which we will precisely define later). This is an obstruction to solving the Fourier uniformity conjecture purely with dynamical methods and without any new input from number theory. \cite{Sawin} further showed that there are dynamical models for the Liouville function which have only polynomially many sign patterns. Explicitly, consider the following function $\tilde a$ which behaves almost like a multiplicative function: we partition the natural numbers into intervals with the length of the intervals slowly tending to infinity. For instance, we could split all the numbers between $10^{10^n}$ and $10^{10^{n+1}}$ into blocks of length $\sim n$. Then on each interval $I$ we pick a random phase $\alpha_I$ in $S^1$  uniformly and indepently. Then we set $\tilde a$ to be the function obtained by rounding the function which sends $n \mapsto e(\alpha_I n)$ for $n$ in $I$. In formulas, we set $\tilde a(n) = 2\ch_{\text{Re } e(\alpha_I n) > 0} -1 $ for $n$ in $I$.  We remark that the dynamical model for this sequence is isomorphic to the product of the dynamical system introduced by \cite{FrantzikinakisHost} with $\widehat{\Z}$ (again, we defer the precise definition until later). Clearly, $\tilde a$ is  not multiplicative. However, it is ``statistically" multiplicative in the sense that, with high probability, for any sign pattern $\epsilon$ of length $k$, for any $m$ and for large $N$
\[
\E^{\log}_{n \leq N} \ch_{\tilde a_{n + h} = \epsilon_h \text{ for all } h \leq k  } \approx \E^{\log}_{n \leq N} m \ch_{m | n} \ch_{\tilde a_{n + mh} = -\epsilon_h \text{ for all } h \leq k  }.
\]
This function clearly does not satisfy the $1$-Fourier uniformity conjecture and \cite{Sawin} showed that it has quadratically many sign patterns that occur with positive upper logarithmic density even though it does satisfy the a version of \cite{MRT}. If we had used a random $\kappa-1$-degree polynomial instead of a random linear polynomial, we would get a function which is again statistically multiplicative but which fails the $\kappa-1$-Fourier uniformity conjecture and \cite{Sawin} showed that it has $\lesssim k^{\kappa + 2 \choose 2}$ many sign patterns of length $k$. However, the author is unaware of any ``dynamical" techiniques that distinguish these statistically multiplicative functions from the Liouville function. This is made precise with Definition \ref{systemdef}.
	
	We give one last application.
	\begin{thm}
		Again, suppose that $a$ is an unpretentious completely multiplicative function taking values in the unit circle. There is a set $C \subset [0,1]$ of Hausdorff dimension 1 such that 
		\[
		\lim_{H \rightarrow \infty} \limsup_{N \rightarrow \infty} \E^{\log}_{n \leq N} \sup_{\alpha \in C} | \E_{h \leq H} a(n + h) e(h\alpha) | = 0.
		\]   
	\end{thm}
	\begin{proof}
		The main idea is to combine Theorem \ref{Cantor} with a diagonalization argument. For a disjoint collection of intervals $\mathcal{J} = \{J\}$ and a natural number $n$ we define $D_n(\mathcal{J})$ to be the set of intervals obtained by taking each $J$, removing a ball of diameter $\frac{|J|}{n}$ around the center of the interval $J$, taking the two remaining intervals, then taking the union over all $J$ in $\mathcal{J}$.
		
		We construct $C$ inductively as follows. Start with any interval $I$ and set $\mathcal{J}_2 = \{ I \}$. Assume inductively that we have constructed $\mathcal{J}_{n-1}$ Then we apply $D_n$ again and again. Let 
		\[
		C_n = \cap_{m \in \N } \cup_{J \in D^m\mathcal{J}_{n-1} } J.
		\]
		Since $C_n$ has box dimension $\frac{\log n-1}{\log n}$, we know by Theorem \ref{Cantor} that there exists a natural number $H_n$ such that if $H \geq H_n$ then 
		\[
		\limsup_{N \rightarrow \infty} \E^{\log}_{n \leq N} \sup_{\alpha \in C} | \E_{h \leq H} a(n + h) e(h\alpha) | \leq \frac{1}{n}.
		\] 
		Then define 
		\[
		\mathcal{J}_n = D_n^{n H_n} \mathcal{J}_{n-1}.
		\]
		We set
		\[
		C = \cap_{n \geq 2} \cup_{J \in \mathcal{J}_n } J.
		\]
		Clearly, the Hausdorff dimension of $C$ is at least $\frac{- \log 2}{\log (\frac{n - 1}{2n}) }$ for every $n$ and therefore the Hausdorff dimension is precisely $1$.
		Now we verify that $C$ has the desired property. For each natural number $m$, by enlarging the set we are maximizing over, we have that 
		\begin{align*}
		&\limsup_{H \rightarrow \infty} \limsup_{N \rightarrow \infty} \E^{\log}_{n \leq N} \sup_{\alpha \in C} | \E_{h \leq H} a(n + h) e(h\alpha) | \\
		\lesssim & 	\limsup_{N \rightarrow \infty} \E^{\log}_{n \leq N} \sup_{\alpha \in J \in \mathcal{J}_m } | \E_{h \leq H_m} a(n + h) e(h\alpha) |. 
		\intertext{Since every element $\alpha \in J \in \mathcal{J}_m$ is in $D_m^{m H_m }(\mathcal{J}_{m-1})$ there exists $\beta = \beta_\alpha$ depending on $\alpha$ such that $\beta$ is in $C_m$ and the distance from $\alpha$ to $\beta$ is no more than $\frac{1}{H_m m}$. Therefore, for all $h \leq H_m$, $\alpha h$ is within $\frac{1}{m}$ of $\beta h$. Thus,   }
		\lesssim & 	\limsup_{N \rightarrow \infty} \E^{\log}_{n \leq N} \sup_{\beta \in C_m} | \E_{h \leq H_m} a(n + h) e(h\beta) | + \frac{1}{m}.
		\intertext{However, by our choice of $H_m$, we have}  
		\lesssim &\frac{1}{m}.
		\end{align*}
		Since $m$ was arbitrary, we obtain the desired result.
	\end{proof}

	\begin{rem}
		We have stated our main theorems in the case that $a$ is completely multiplicative and takes values in the unit circle. We remark that these assumptions can be weakened to include all multiplicative functions taking values in the unit disk. The reduction from multiplicative functions taking values in the unit disk to multiplicative functions taking values in the unit circle is essentially due to Tao (see \cite{TaoChowla}, Proposition 2.1). The reduction from multiplicative functions to completely multiplicative functions (say both taking values in the unit circle) is carried out in the second appendix to this paper. The argument is rather short and was essentially communicated to me by Tao. However, it may be more broadly known and I make no claim of originality.
	\end{rem}
	
	We now sketch an outline of an argument that is morally very similar to the main argument in this paper. However, for the moment we will work in a more concrete setting. To make this argument rigorous, it is much easier to pass to the dynamical context. Suppose that $b$ is a sequence with quadratic word growth rate and that
	\[
	\limsup_{N \rightarrow \infty} | \E^{\log}_{n \leq N} \lambda(n) b(n) | > c > 0.
	\]
	Then we can fix a natural number $k$ and average over translates,
	\[
	\limsup_{N \rightarrow \infty} | \E^{\log}_{n \leq N} \E_{h \leq k} \lambda(n + h) b(n + h) | > c.
	\]
	Fix a large natural number $P$ with $N \gg P \gg k$. Because $\lambda$ also has a multiplicative symmetry, we can average over dilates 
	\[
	\limsup_{N \rightarrow \infty} | \E^{\log}_{n \leq N} \E_{P/2 < p \leq P} \E_{h \leq k} \lambda(pn + ph) b(n + h) | > c.
	\]
	Moving the absolute values inside and crudely replacing $b$ by the worst word of length $k$, we get 
	\[
	\limsup_{N \rightarrow \infty}  \E^{\log}_{n \leq N} \sup_{\epsilon} \E_{P/2 < p \leq P} | \E_{h \leq k} \lambda(pn + ph) \epsilon_h | > c,
	\]
	where the supremum is taken over all words $\epsilon$ of $b$. Tao's entropy decrement argument, introduced in \cite{TaoChowla}, allows us to replace $pn$ by $n$.
	\[
	\limsup_{N \rightarrow \infty}  \E^{\log}_{n \leq N} \sup_{\epsilon} \E_{P/2 < p \leq P} | \E_{h \leq k} \lambda(n + ph) \epsilon_h | > c.
	\]
	Now if $\lambda$ behaves randomly, then we already know that $\lambda$ is orthogonal to $b$. Therefore, if $\lambda$ correlates with $b$ it must have some structure. Morally, \cite{FrantzikinakisHost} says we can break up $\lambda$ into a structured part and a random part, and that all the correlation comes from the structured part. \cite{HostKra} proves that the structured part must take the form of a nilsequence. For the purposes of this sketch, we will focus on the case that there exists $\alpha_n$ and $\beta_n$ irrational such that
	\[
	\limsup_{N \rightarrow \infty}  \E^{\log}_{n \leq N} \sup_{\epsilon} \E_{P/2 < p \leq P} | \E_{h \leq k} e(\alpha_n (ph)^2 + \beta_n ph ) \epsilon_h | > c.
	\]
	By H\"older's inequality,
	\[
	\limsup_{N \rightarrow \infty}  \E^{\log}_{n \leq N} \sup_{\epsilon} \E_{P/2 < p \leq P} | \E_{h \leq k} e(\alpha_n (ph)^2 + \beta_n ph ) \epsilon_h |^4 > c^4.
	\]
	By the pigeonhole principle, since there are only $\delta k^2$ many sign patterns, there is a sign pattern $\epsilon$ such that,
	\[
	\limsup_{N \rightarrow \infty}  \E^{\log}_{n \leq N} \E_{P/2 < p \leq P} | \E_{h \leq k} e(\alpha_n (ph)^2 + \beta_n ph ) \epsilon_h |^4 > \delta^{-1} k^{-2} c^4.
	\]
	Expanding everything out and using that $\delta \leq \frac{c^{4}}{2}$
	\[
	\limsup_{N \rightarrow \infty}  \E^{\log}_{n \leq N} | \E_{P/2 < p \leq P}   \E_{j \in [k]^4} e(\alpha_n p^2(j_1^2 + j_2^2 - j_3^2 - j_4^2) + \beta_n p(j_1 + j_2 - j_3 - j_4) ) | >  2k^{-2}.
	\]
	When $j_1^2 + j_2^2 - j_3^2 - j_4^2 \neq 0$ or $j_1 + j_2 - j_3 - j_4 \neq 0$ then for $P$ large, by the circle method
	\[
	\E_{P/2 < p \leq P} e(\alpha_n p^2(j_1^2 + j_2^2 - j_3^2 - j_4^2) + \beta_n p(j_1 + j_2 - j_3 - j_4) )  \approx 0.
	\]
	The analogue of the circle method for more general nilpotent Lie groups was introduced in \cite{GT1}, \cite{GT2} and \cite{GTZ}.
	The analogue of the step where we conclude that the sums of powers is $0$ for more general nilpotent Lie groups is an argument of \cite{Frantzikinakis}.
	Thus the only contribution is from the terms where  $j_1^2 + j_2^2 - j_3^2 - j_4^2 = 0$ and $j_1 + j_2 - j_3 - j_4 = 0$. But it is easily seen from Newton's identities for symmetric polynomials that this only happens for the $2k^2$ ``diagonal" terms. Thus, we get
	\[
	2k^{-2} > 2k^{-2},
	\]
	which of course provides a contradiction. For the proof of Theorem \ref{condmainthm}, we need to not only use the theory of symmetric polynomials but also use \cite{MR3548534}.

	\subsection{Background and notation}\label{background}
	Suppose $a(n)$ is a 1-bounded, unpretentious multiplicative function with $|a(n)| = 1$ for all $n$. Let $b(n)$ a sequence where only $o(k^2)$ or $O(k^{t - \e})$ many sign patterns occur with positive $\log$-density. The usual construction of a Furstenberg system (see \cite{MR670131}) for $(a,b)$ proceeds as follows: consider the point $(a, b)$ in the space of pairs of sequences. Then apply a random shift to this deterministic variable, $(T^n a, T^n b)$. This gives a random variable in the space of pairs of sequences. The distribution of this random variable is then a shift invariant measure on the space of pairs of sequences. Furthermore, if $f$ is the function on the space of pairs of sequences that evaluates the first sequence at $1$ and $f'$ is the function which evaluates the second sequence at $1$ then
	\[
	f(T^n a, T^n b ) f' (T^n a, T^n b) = a(n+1) b(n+1),
	\] 
	which is the sequence whose average value we care about. Therefore, if the average of $a(n) b(n)$ is greater than $c$ in absolute value then
	\[
	\left| \int f \cdot f' \right| > c,
	\]
	as well. Of course, it does not really make sense to take a random natural number. Instead, one must shift by a random natural number in a large but finite interval whose length tends to infinity, then find a subsequence of the random variables that converges in distribution. This corresponds to taking a weak-$\ast$ limit of the corresponding measures.
	
	However, we take a slightly modified approach. The reason is that the function $a$ has some additional symmetry, namely $a(nm) = a(n)a(m)$. As such, the probability that some word occurs i.e., that $a(n + h) = \epsilon_h$ for $h = 1, \ldots, k$ and for $n$ randomly chosen between $1$ and $N$ is the same as the probability that $a(pn + ph) = a(p) \cdot \epsilon_h$ for $h = 1, \ldots, k$ and for $n$ chosen randomly between $1$ and $N$. That's the same as $p$ times the probability that for a randomly chosen $n$ between $1$ and $pN$ one has $a(n + ph) = a(p) \cdot \epsilon_h$ for $h = 1, \ldots, k$ and $p$ divides $n$. Just flipping everything around, the probability that a random $n$ between $1$ and $pN$ satisfies $a(n + ph) = a(p) \cdot \epsilon_h$ and is divisible by $p$ is $\frac{1}{p}$ times the probability that a random $n$ between $1$ and $N$ satisfies $a(n + h) = \epsilon_h$. We want our dynamical system to capture this symmetry. There are two difficulties which arise when we want to translate this symmetry to our dynamical system. The first is that the interval keeps changing: the distribution of $T^n a$ might be very different on the intervals from $1$ to $N$ and from $1$ to $pN$ so when we take a weak limit along a subsequence of intervals, the distribution $T^n a$ for shifts in one interval might approximate our invariant measure while shifts along the other interval might not. The fix for this problem is to use $\log$-averaging. After we weight each natural number $n$ by $\frac{1}{n}$, the probability that a random $n$ will be between $N$ and $pN$ is $\sim \frac{\log p}{\log N}$ which tends to $0$ as $N$ tends to infinity. Therefore, the distribution of $T^n a$ for a random $n$ between $1$ and $N$ is very close to the distribution of $T^n a$ for a random $n$ between $1$ and $pN$ as long as we choose $n$ randomly using logarithmic weights. The other problem is that our dynamical system does not have a good notion of ``being divisible" by a number. To remedy this, we make use of the profinite completion of the integers
	\[
	\widehat{\Z} = \prod_{p} \Z_p,
	\]
	where $p$ is always restricted to be a prime and $\Z_p$ is the $p$-adic integers $\Z_p = \lim_{\leftarrow} \Z / p^k \Z$ i.e. the inverse limit of $\Z / p^k \Z$ for all $k$. For each natural number $n$, we get an element of $\widehat{\Z}$ by reducing $n \mod p^k$ for every prime $p$ and every natural number $k$. Then to build our dynamical system, we take the space of triples consisting of two sequences and a profinite integer and for a logarithmically randomly chosen integer $n$ we consider the random variable $(T^n a, n, T^n b)$ in this space. The distribution of this random variable is a shift invariant measure. Furthermore, we have the following symmetry: let $Y = \overline{\{ T^n b \colon n \in \N \}}$ and $X = (S^1)^\N \times \widehat{\Z}$. Define the function
	\[
	M \colon X \rightarrow \widehat{\Z},
	\]
	by projecting onto the $\widehat{\Z}$ coordinate in $X$,
	\[
	M \colon (\alpha, r) \mapsto r.
	\]
	Define the function
	\[
	I_p \colon M^{-1} (p \widehat{\Z} ) \rightarrow X, 
	\]
	by ``zooming in" by a factor of $p$ and multiplying by $\overline{a(p)}$ on the first factor and dividing by $p$ on the second,
	\[
	I_p \colon (\alpha(n), r ) \mapsto (\overline{a(p)} \alpha(pn), r/p ),
	\]
	where $r/p$ is the unique element of $\widehat{\Z}$ such that $p \cdot (r/p) = r$.
	Then if $\nu$ is our invariant measure on $X \times Y$ and $\mu$ is its first marginal then $I_p$ pushes forward $\mu$ restricted to $M^{-1}(p \widehat{\Z})$ to $\frac{1}{p} \mu$. Formally, we make the following definition:
	\begin{defn}\label{systemdef}
		Let $(X, \mu, T)$ be a dynamical system, let $f \colon X \rightarrow \C$ be a measurable function, let $M \colon X \rightarrow \widehat{\Z}$ be a measurable function and for each $m$ let $I_m \colon M^{-1}(m \widehat{\Z}) \rightarrow X$ be a measurable function. We say $(X, \mu, T, f, M, I_m)$ is a dynamical model for $a$ if,
		\begin{itemize}
			\item $M \circ T = M + 1$ almost everywhere.
			\item $I_m \circ T^m = T \circ I_m$ almost everywhere in $M^{-1}(m \widehat{\Z})$.
			\item $I_m$ pushes forward the measure $\mu$ restricted to $M^{-1}(m \widehat{\Z})$ to $\frac{1}{m} \mu$. Symbolically, for any function $\phi$ in $L^1(\mu)$ we have
				\[
				\int_X \phi(x) \mu(dx) = \int_X m \ch_{x \in M^{-1} (m \widehat{\Z} )} \phi ( I_m (x) ) \mu(dx).
				\]
			\item $f \circ I_m = \overline{a(m)} \cdot f$ almost everywhere in $M^{-1}(m \widehat{\Z})$.
			\item For all $m$ and $n$, $I_{nm} = I_n \circ I_m$ almost everywhere in $M^{-1}(mn \widehat{\Z})$.
		\end{itemize}
		We also ask for the following property that \cite{Sawin} does not impose.
		\begin{itemize}
			\item For any natural number $m$ and any measurable subset $A$ of $\C^m$,
			\begin{align*}
			\mu&\{ x \in X \colon (f(T^1 x), \ldots, f(T^m x)) \in A \} \leq \\  \logd &\{ n \leq N \colon (a(n+1), \ldots, a(n+m)) \in A \},
			\end{align*}
where $\logd$ denotes upper logarithmic density.
			We remark that we can also fix a Banach limit $p-\lim$ extending the usual limit functional and require that equality holds in the previous equation holds for any limit taken with respect to that Banach limit. For more details, see \cite{TaoBlog}.
		\end{itemize}
		Let $(X \times Y, \nu, T)$ be a joining of two dynamical systems $X$ and $Y$. Suppose that $\mu$ is the first marginal and $(X, \mu, T, f, M, I_m)$ is a dynamical model for $a$. Let $f'$ be a measurable function on $X \times Y$ which is $Y$ measurable. We say $(X \times Y, \nu, T, f, f', M, I_m)$ is a joining of a dynamical model of $a$ with $b$ if we also have that, for any natural number $m$ and any measurable subset $A$ of $\C^m$,
		\begin{align*}
		\nu&\{ (x,y) \in X \times Y \colon (f'(T^1 y), \ldots, f'(T^my)) \in A \} \leq \\  \logd &\{ n \leq N \colon (b(n+1), \ldots, b(n+m)) \in A \}
		\end{align*}
where $\logd$ denotes upper logarithmic density.
		We could also require that the joint statistics of $(f, f')$ agree with the joint statistics of $(a,b)$ but this is not necessary for our argument.
	\end{defn} 
	\begin{rem}
		The preceding definition was used first in \cite{TaoBlog} and generalized in \cite{Sawin}.
	\end{rem}
We abuse notation and denote all transformations by the letter $T$. We also remark that for the proof of Theorems \ref{mainthm} and \ref{condmainthm} that $f'$ only takes finitely many values.

	We now specify some notation used in the main argument:
	\begin{itemize}
		\item We fix an unpretentious 1-bounded multiplicative function $a$. (For the definition of unpretentious, see \cite{MRT}; we will only really use that $a$ is unpretentious in Theorem \ref{MRT}; we remark that the Liouville function is unpretentious). We fix constants $t \in \N$, $c > 0$ and $\delta > 0$. We fix a 1-bounded function $b$ with at most $o(k^2)$ or $ k^{t - \e}$ many words of length $k$ occurring with positive upper logarithmic density for all $k \in \K$ where $\K$ is some fixed infinite set. We suppose that
		\[
		\limsup_{N \rightarrow \infty} |\E^{\log}_{n \leq N} a(n) b(n) | > c.
		\]
		We fix $\eta > 0$ such that 
		\[
		\limsup_{N \rightarrow \infty} |\E^{\log}_{n \leq N} a(n) b(n) | > c + \eta.
		\]
		\item 
		We use the following theorem of \cite{FrantzikinakisHost2}.
		\setcounter{section}{2}
		\setcounter{thm}{14}
		\begin{thm}[\cite{FrantzikinakisHost2} Theorem 1.5]
			There exists a joining of a dynamical model for $a$ with $b$, $(X \times Y, T, \nu, f, f', M, I_m)$ 
			satisfying
			\[
			\left| \int_{X \times Y} f(x,y) f'(x,y) \nu(dx dy) \right| > c + \eta,
			\]
			and if $\mu$ is the first marginal then the ergodic components $(X, \mu_\omega, T)$ are isomorphic to products of Bernoulli systems with the Host Kra factor of $(X, \mu_\omega, T)$.
		\end{thm}
		\setcounter{section}{1}
		\setcounter{thm}{16}
		Because the statement here is slightly different than Theorem 1.5 in \cite{FrantzikinakisHost2}, we will go through the details in the first appendix. We fix such a system.
		We will always denote by $\mu$ the first marginal of $\nu$. We also fix ergodic decompositions $\nu = \int_\Omega \nu_\omega d\omega$ and $\mu = \int_\Omega \mu_\omega d\omega$.
		We define the words of length $k$ of $f'$ to be those words $\epsilon$ of length $k$ such that the set of $(x,y)$ such that $f'(T^h x,T^h y) = \epsilon_h$ for all $h \leq k$ has positive measure. We note that the set of words of $f'$ is a subset of the set of words of $b$ that occur with positive upper $\log$ density: after all, if $f'(T^h x,T^h y) = \epsilon_h$ then by definition of a joining of a dynamical model $a$ with $b$,
\begin{align*}
0 <& \mu \{ (x, y) \in X \times Y \colon f'(T^h y) = \epsilon_h \text{ for all } h  \leq k \} \\  \leq& \logd \{ n \in \N \colon b(n+h) = \epsilon_h \text{ for all } h \leq k \},
\end{align*}
where $\logd$ denotes upper logarithmic density.
		\item $G$ will always refer to a nilpotent Lie group. $G_s$ will always refer to the $s^{th}$ step in the lower central series. $\Gamma$ will always refer to a cocompact lattice in $G$, meaning that $G_s / \Gamma_s$ is compact for every $s$. $g, \sigma$ and $\tau$ will always refer to group elements. $\mathcal{B}$ will always refer to the Borel sigma algebra. We will fix a particular $G$, $\Gamma$ and $g$ following Corollary \ref{fixG}. For more on this see \cite{GT3}.
		\item For a nonempty, finite set $A$ and $\phi \colon A \rightarrow \C$, we denote $\E_{n \in A} \phi(n) = \frac{1}{\# A}\sum_{n \in A} \phi(n)$. For $A \subset \N$, we denote \[ \E^{\log}_{n \in A} \phi(n) = \frac{1}{\sum_{n \in A} \frac{1}{n}}\sum_{n \in A} \frac{\phi(n)}{n}. \] This notation is due to Frantzikinakis (see \cite{Frantzikinakis}). We always restrict $p$ to be prime.
		By definition a nilsystem is a dynamical system $(G/\Gamma, dx, T, \mathcal{B})$ where $G$ is a nilpotent Lie group, $\Gamma$ is a cocompact subgroup, $dx$ is Haar measure, there exists $g$ such that $T(x) = gx$ and $\mathcal{B}$ is the Borel sigma algebra. A nilsequence is a sequence of the form $F(g^n \Gamma)$ where $G$ is a nilpotent Lie group, $\Gamma$ is a cocompact lattice in $G$, $g$ is an element in $G$ and and $F \colon G / \Gamma \rightarrow \C$ is a continuous function. Suppose $G$ is an $s$-step nilpotent Lie group so that $G_s$ is an abelian group and $G_s / \Gamma_s$ is a compact abelian group. Then a nilcharacter $\F$ is a function $G / \Gamma \rightarrow \C$ such that there exists a character $\xi \colon G_s /\Gamma_s \rightarrow S^1$ called the frequency of $\F$ such that, for all $x$ in $G /\Gamma$ and $u$ in $G_s$ we have $\F(u x) = \xi(u \Gamma_s) \F(x)$. We will abuse notation and identify $\xi$ with the function on $G_s$ that maps $u \mapsto \xi(u \Gamma_s )$. We say $\xi$ is nontrivial if there exists $u$ in $G_s$ such that $\xi(u) \neq 1$. We say $\xi$ is nontrivial on the identity component if we can find a $u$ in the identity component of $G_s$ such that $\xi(u) \neq 1$.
		\item For Theorem \ref{tao}, we will adopt conventions from the theory of Shannon entropy. In particular, $H(x)$ will denote the Shannon entropy of $x$ and $I(x,y)$ will denote the mutual information between $x$ and $y$. For more details, see \cite{TaoChowla}.
		\item We will always denote by $\z$ the smallest sigma algebra on $X$ generated by the union of the sigma algebras corresponding each of the Host Kra factors. We will denote \[B = \{ (x,y) \in X \times Y \colon f'(T^n y) \text{ is eventually periodic as a function of $n$} \}. \] Since whether $(x,y)$ is in $B$ depends only $y$, we will abuse notation and also use  $B = \{ y \in Y \colon f'(T^n y) \text{ is eventually periodic as a function of $n$} \}$.
		\item For a complex numbers $z$, a set $A$ and a real number $w$ we say $z = O_A(w)$ and $z \lesssim_A w$ if there exists a constant $C$ depending on $A$ but not $z$ and $w$ such that $|z| \leq C w$. If there are more subscripts we mean that the constant may depend on more parameters. For instance, by $\lesssim_{A,u,K}$ we mean that the implied constant can depend on $A$, $u$ and $K$. 
	\end{itemize}
	\subsection{Acknowledgments}
	Special thanks to Terence Tao for sharing many ideas on an earlier version of this paper and for his many helpful comments. Also, special thanks to Nikos Frantzikinakis for pointing out a number of ways to strengthen the main theorem of this paper. I would also like to thank Tim Austin, Bj\"orn Bringmann, Alex Dobner, Asgar Jamneshan, Bernard Host, Gyu Eun Lee, Zane Li, Adam Lott, Maksym Radziwi\l\l, Bar Roytman, Chris Shriver, Joni Ter\"av\"ainen and Alex Wertheim for many helpful discussions. Thanks to Will Sawin for suggesting I use Vinogradov's mean value theorem to improve an earlier version of Theorem \ref{condmainthm}. I would lastly like to thank the anonymous reviewer for their extremely helpful comments. Some of this work was completed while the author was at the American Institute for Mathematics workshop on Sarnak's conjecture.

	\section{Main Argument}
	In this section, we prove Theorem \ref{mainthm} and Theorem \ref{condmainthm}. In Section \ref{section3}, we explain how to adapt the proof to handle Theorem \ref{approxmainthm}.
	
	We remark that much of the notation, including $X, Y, \mu, \nu, f, f', a,$ and $b$ was defined in Subsection \ref{background}.

	We start off with a theorem by \cite{MRT}, relying on work in \cite{MR}. This is a special case of our theorem, so it is no surprise that we need this result.
	\begin{thm}\label{MRT}[\cite{MRT} Theorem 1.7; see also \cite{MR}]
		Let $a$ be a bounded, non-pretentious multiplicative function. Let $\theta$ be a periodic sequence. Then
		\[
		\lim_{H \rightarrow \infty} \limsup_{N \rightarrow \infty} \E^{\log}_{n \leq N} | \E_{h \leq H} a(n+h) \theta(h) | = 0.
		\]
	\end{thm}
	This theorem says that $a$ does not locally correlate with periodic functions.
	Eventually, we plan to use a local argument. In particular, our argument will only work for those points where $f'$ does not behave locally like a periodic function. Therefore, we need to exclude any contribution to the integral coming from points where $f'$ behaves like a periodic function. That is the content of the following corollary.
	\begin{cor}\label{periodicpoints}
		Let $B = \{ (x,y) \in X \times Y \colon f'(T^n y) \text{ is eventually periodic as a function of $n$} \}$. Then
		\[
		\int_{B} f(x) f'(y) \nu(dx dy) = 0.
		\]
	\end{cor}
	\begin{proof}
		In this proof, we introduce some notation which will not be used in the rest of the paper.
		Because $T$ preserves $\nu$ and because $B$ is $T$-invariant, we can average over shifts:
		\[
		\int_{B} f(x) f'(y) \nu(dx dy) = \lim_{H \rightarrow \infty} \int_{B} \E_{h \leq H} f(T^h x) f'(T^h y) \nu(dx dy).
		\]
		We know $f'$ takes only finitely many values. There are only countably many different periodic sequences taking values in a finite alphabet. Therefore, it suffices to prove that if $B_\theta$ is the set of points $(x,y)$ on which $f'(T^h y)$ is eventually equal to the periodic function $\theta$ that 
		\begin{align}
		0 =& \lim_{H \rightarrow \infty} \int_{B_\theta} \E_{h \leq H} f(T^h x) f'(T^h y) \nu(dx dy). \nonumber
		\intertext{
			Let $\e > 0$. By Theorem \ref{MRT}, for $H$ sufficiently large 
		}
		\e^3 & \gg \limsup_{N \rightarrow \infty} \E^{\log}_{n \leq N} \sup_{j \in \N} | \E_{h \leq H} a(n+h) \theta(h+j) |. \label{ppeq}
		\intertext{
We claim that, translating this to the dynamical world using the definition of a dynamical model for $a$,
		}
		\nu \{ (x, y) & \ \colon \limsup_{H \rightarrow \infty}  | \E_{h \leq H} f(T^h x) \theta(h) | > \e \} \leq \e. \nonumber
		\intertext{
After all, by Chebyshev's inequality, for any $H$ such that \ref{ppeq} holds,
}
\logd \{ n \in \N & \ \colon \sup_{j \in \N} | \E_{h \leq H} a(n+h) \theta(h+j) | \geq \e \} \ll \e^2, \nonumber
\intertext{
where $\logd$ denotes upper logarithmic density. Fix such an $H$ for the moment and fix a natural number $H' > H$. In fact, more is true. Let $S$ be the subset of the natural numbers such that $n$ is in $S$ if and only if there exists a natural number $H' \geq H_n \geq H$ such that 
}
& \ | \E_{h \leq H_n} a(n+h) \theta(h) | \geq \e. \nonumber
\intertext{
Let $S'$ denote the union of all the intervals $[n+1, n+H_n]$ for $n$ in $S$. We claim there is a subcollection $\mathcal{I}$ of these intervals which covers $S'$ and such that each natural number is contained in at most two intervals in $\mathcal{I}$. This is a somewhat standard covering lemma, but we include the details for the interested reader. For instance, consider the following construction. Let $\mathcal{I}_0$ denote the empty set. Then assuming we have constructed $\mathcal{I}_\ell$ for some natural number $\ell$, let $m$ denote the smallest natural number in $S'$ not contained in the union of the intervals in $\mathcal{I}_\ell$. (If no such $m$ exists, then just set $\mathcal{I}_{\ell+1} = \mathcal{I}_\ell$). Let $n$ be a natural number maximizing $n + H_n$ subject to the constraints that $n$ is in $S$ and $m$ is in $[n+1, n+H_n]$. Such an $n$ exists because $m$ is in $S'$. Then let $\mathcal{I}_{\ell+1} = \mathcal{I}_\ell \cup \{ [n+1, n+H_n] \}$. Now we check that $\mathcal{I} = \cup \mathcal{I_\ell}$ has the desired property. First, for any $m$ in $S$, $m$ is clearly contained in the union of the intervals in $\mathcal{I}_m$. Thus, $\mathcal{I}$ covers $S'$. Second, suppose that $m$ in $S'$ is contained in $I_1$, $I_2$ and $I_3$ with $I_\ell$ chosen before $I_{\ell + 1}$ for $\ell = 1,2$. Suppose that $\mathcal{I}_{\ell_i}$ is the first set of the form $\mathcal{I}_\ell$ where $I_i$ is contained in $\mathcal{I}_{\ell_i}$. Then the union of the intervals in $\mathcal{I}_{\ell_1}$ contains $m$. Thus, there exists $m'$ in $S'$ such that $I_2 = [n+1, n + H_n]$ was chosen to maximize $n + H_n$ subject to the constraint that $m'$ is in $I_2$. Since we assumed $m$ was contained in $I_2$, we have that $n < m$. Now let $I_3 = [n' +1, n' + H_{n'}]$. Since $I_3$ also contains $m$, $n'$ is also less than $m$ which is in turn less than $m'$. But $I_2$ maximized $n + H_n$ over all intervals containing $m'$ and if $n' + H_{n'}$ were larger than $n + H_n$ which is larger than $m'$, then $I_3$ would contained $m'$ as well. Thus $n' + H_{n'} \leq n + H_n$ and therefore any point contained in $I_3$ is already contained in the union of the intervals in $\mathcal{I}_{\ell_2}$. Therefore, $I_3$ should not have been selected for $\mathcal{I}_{\ell_3}$ which leads to a contradiction. Thus, every natural number is covered at most twice by the union of the intervals in $\mathcal{I}$. If 
}
& \ | \E_{h \leq H_n} a(n+h) \theta(h) | \gtrsim \e \nonumber
\intertext{
then
}
& \ \E_{h \leq H_n} | \E_{h' \leq H}  a(n+h +h') \theta(h+h') | \gtrsim \e. \nonumber
\intertext{
Therefore, for at least $\e \cdot H_n$ many points $n+h'$ in the interval $[n+1, n+H_n]$, 
}
\sup_{j \in \N} & | \E_{h \leq H} a(n+h+h') \theta(h+j) | \gtrsim \e. \nonumber
\intertext{
However, we know that
}
\logd \{ n \in \N & \ \colon \sup_{j \in \N} | \E_{h \leq H} a(n+h) \theta(h+j) | \geq \e \} \ll \e^2, \nonumber
\intertext{
and that each such natural number is contained in at most two intervals of the form $[n+1, n+H_n]$ in $\mathcal{I}$. We conclude that, by Chebyshev's inequality, the logarithmic density of $S'$ is at most $\e$. Therefore, the logarithmic density of $S$ is at most $\e$. This precisely means
}
\logd \{ n \in \N & \ \colon \sup_{L \in [H, H']} | \E_{h \leq L} a(n+h) \theta(h) | \geq \e \} \leq \e. \nonumber
\intertext{
The condition $\sup_{L \in [H, H']} | \E_{h \leq L} a(n+h) \theta(h) | \geq \e$ depends measurably on $(a(n+1), \ldots, a(n+H'))$ so by definition of a dynamical model for $a$,
}
\nu \{ (x,y) & \ \colon \sup_{L \in [H, H']} | \E_{h \leq L} f(T^h x) \theta(h) | \geq \e \} \leq \e. \nonumber
\intertext{
Since this is true for all $H'$, we get that
}
\nu \{ (x,y) & \ \colon \sup_{L \geq H} | \E_{h \leq L} f(T^h x) \theta(h) | \geq \e \} \leq \e. \nonumber
\intertext{
Since this is true for all $\e > 0$, for all $(x, y)$ outside a set of measure $0$, we have 
}
\lim_{H \rightarrow \infty} & \   \E_{h \leq H} f(T^h x) \theta(h)  = 0. \nonumber
\intertext{
 For $(x,y) \in B_\theta$, we know that $f'(T^h y) = \theta(h)$ for $h$ sufficiently large so for $(x,y) \in B_\theta$ outside a set of measure $0$, we know $\E_{h \leq H} f(T^h x) f'(T^h y) \rightarrow 0$. By the dominated convergence theorem, we have
		}
		0 =& \lim_{H \rightarrow \infty} \int_{B_\theta} \E_{h \leq H} f(T^h x) f'(T^h y) \nu(dx dy) \nonumber
		\end{align}
		as desired.
	\end{proof}
	We will also need the following result later. It states that $f$ does not correlate locally with periodic functions.
	\begin{cor}\label{MRTerg}
		Let $\mu = \int_\Omega \mu_\omega d\omega$ be an ergodic decomposition of $\mu$.
		For almost every $\omega$, for all $1$-bounded function $\phi \colon X \rightarrow \C$ such that, for $\mu_\omega$ almost every $x$, $\phi(T^hx)$ is periodic in $h$ we have
		\[
		\int_X f(x) \phi(x) \mu_\omega(dx) = 0.
		\]
	\end{cor}
	\begin{proof}
		Let $d$ be a natural number. Then we claim that, 
		\[
		\lim_{H \rightarrow \infty} \lim_{N \rightarrow \infty} \E^{\log}_{n \leq N} \sup_{\theta \in S_d } | \E_{h \leq H} a(n+h) \theta(h) | = 0,
		\]
		where $S_d$ is the set of $d!$ periodic, 1-bounded functions. Since the supremum is over a finite set, this directly follows from Theorem \ref{MRT}. Let $\e> 0$. For $H$ sufficiently large,
		\[
		\limsup_{N \rightarrow \infty} \E^{\log}_{n \leq N} \sup_{\theta \in S_d } | \E_{h \leq H} a(n+h) \theta(h) | \leq \e^3.
		\]
		Therefore, as in the proof of Proposition \ref{periodicpoints}
		\[
		\nu \{ (x, y) \ \colon \limsup_{H \rightarrow \infty}  \sup_{\theta \in S_d} | \E_{h \leq H} f(T^h x) \theta(h) | > \e \} \leq \e.
		\]
		Since this is true for all $\e$, we get that
		\[
		\int_X \limsup_{H \rightarrow \infty} \sup_{\theta \in S_d} | \E_{h \leq H} f(T^h x) \theta(h) | \mu(dx) = 0.
		\]
		Therefore, there exists $\Omega_d \subset \Omega$ of full measure such that for $\omega$ in $\Omega_d$, we have
		\[
		\int_X \limsup_{H \rightarrow \infty} \sup_{\theta \in S_d} | \E_{h \leq H} f(T^h x) \theta(h) | \mu_\omega (dx) = 0.
		\]
		Now let $\omega$ be an element of $\Omega_d$ for all $d$ and let $\phi$ be a 1-bounded function such that $\phi(T^h x)$ is periodic in $h$ for $\mu_\omega$-almost every $x$. Suppose that there exists $\e > 0$ such that
		\[
		\Big| \int_X f(x) \phi(x) \mu_\omega(dx) \Big| > \e.
		\]
		Then by translation invariance, we know
		\[
		\Big| \limsup_{H \rightarrow \infty} \int_X \E_{h \leq H} f(T^h x) \phi(T^h x) \mu_\omega(dx) \Big| \geq \e.
		\]
		Let $X_d$ be the set of all points $x$ such that $\phi(T^h x)$ is periodic with period at most $d$. Note by assumption that $\mu_\omega(\cup X_d) = 1$. Then by dominated convergence, there exists $d$ such that
		\[
		\int_{X_d}  \limsup_{H \rightarrow \infty} | \E_{h \leq H} f(T^h x) \phi(T^h x) | \mu_\omega (dx) > .5 \e.
		\]
		Since $\phi(T^h x)$ is $d!$ periodic for every $x$ in $X_d$, this integral is bounded by 
		\[
		\int_X \limsup_{H \rightarrow \infty} \sup_{\theta \in S_d} | \E_{h \leq H} f(T^h x) \theta(h) | \mu_\omega (dx),
		\]
		which gives a contradiction.
	\end{proof}
For the proof of Theorem \ref{condmainthm}, we also need an upgraded version of Corollary \ref{MRTerg} under the assumption that the $\kappa-1$-Fourier uniformity conjecture holds. 
\begin{prop}\label{conderg}
Suppose that the $\kappa-1$-Fourier uniformity conjecture holds i.e., for every nilpotent Lie group $G$ of step $< \kappa$, every cocompact lattice $\Gamma$ and every continuous function $F \colon G / \Gamma \rightarrow \C$
\[
\lim_{H \rightarrow \infty} \limsup_{N \rightarrow \infty} \E^{\log}_{n \leq N} \sup_{g \in G} | \E_{h \leq H} a(n+h) F(g^h \Gamma) | = 0.
\]
Then for almost every $\omega$,  we have the following property: for every nilpotent Lie group $G$ of step $< \kappa$, every cocompact lattice $\Gamma$, every continuous function $F \colon G / \Gamma \rightarrow \C$ and every function $\phi$ on $X$ such that for $\mu_\omega$ almost every $x$ there exists $x'$ in $G / \Gamma$ and $g$ in $G$ we have $\phi(T^h x) = F(g^h x')$ for all $h$ in $\N$ we have that
\[
\int_X f(x) \phi(x) \mu_\omega (dx) = 0.
\]  
\end{prop}
\begin{proof}
In the proof of this proposition, we will introduce some notation which will not be used in the rest of the paper.
By, for instance, \cite[Chapter 10, Theorem 28]{HostKra2} there are only countably many pairs $(G, \Gamma)$ up to isomorphism of $G / \Gamma$. Thus, we can fix a sequence $(G_i, \Gamma_i)$ of nilpotent Lie groups of step $< \kappa$ and cocompact lattices such that, for any nilpotent Lie group $G$ of step $< \kappa$ and for any cocompact lattice $\Gamma$ there exists a natural number $i$ and a Lie group isomorphism $\psi \colon G_i \rightarrow G$ such that $\psi(\Gamma_i) = \Gamma$. By Stone-Weierstass, there exists a countable, uniformly dense subset of the continuous functions on $G_i / \Gamma_i$. Fix such a subset and call it $\funs_i$. 
We are assuming the $\kappa-1$-Fourier uniformity conjecture:
\begin{align*}
\lim_{H \rightarrow \infty} \limsup_{N \rightarrow \infty} \E^{\log}_{n \leq N} &\sup_{g \in G_i} | \E_{h \leq H} a(n+h) F(g^h \Gamma_i) | = 0,
\intertext{for all $i$ and all $F$ a continuous function on $G / \Gamma$.
By \cite[Section 4.5, Step 4]{Frantzikinakis} we also get that, for all $i$ and $F$ as before,}
\lim_{H \rightarrow \infty} \limsup_{N \rightarrow \infty} \E^{\log}_{n \leq N} &\sup_{\substack{ g \in G_i \\ x \in G_i / \Gamma_i }} | \E_{h \leq H} a(n+h) F(g^h x) | = 0.
\intertext{Fix a natural number $i$ for the moment and a function $F$ in $\funs_i$.
For each $\e > 0$ there exists $H_\e$ such that }
\limsup_{N \rightarrow \infty} \E^{\log}_{n \leq N} &\sup_{\substack{ g \in G_i \\ x \in G_i / \Gamma_i }} | \E_{h \leq H_\e} a(n+h) F(g^h x) | \ll \e^3.
\intertext{Therefore, by Chebyshev's inequality,}
\logd \{ n \in \N \colon  &\sup_{\substack{ g \in G_i \\ x \in G_i / \Gamma_i }} | \E_{h \leq H_\e} a(n+h) F(g^h x) | \geq \e \} \ll \e^2,
\intertext{where $\logd$ denotes the upper logarithmic density. Note that }
&\sup_{\substack{ g \in G_i \\ x \in G_i / \Gamma_i }}  | \E_{h \leq H_\e} a(n+h) F(g^h x) |
\end{align*}
depends measurably on $(a(n+1), \ldots, a(n+H_\e))$. Thus, there exists some set $A$ in $\C^{H_e}$ such that 
\begin{align*}
 \sup_{\substack{ g \in G_i \\ x \in G_i / \Gamma_i } }| \E_{h \leq H_\e} a(n+h) F(g^h x) | &\geq \e 
\intertext{if and only if $(a(n+1), \ldots, a(n+H_\e))$ are in $A$. Therefore, we know that }
\logd \{ n \in \N \colon  (a(n+1), \ldots, a(n+H_\e)) \in A \} &\ll \e^2.
\intertext{By the definition of a dynamical model for $a$,}
\mu \{ x' \in X \colon  (f(T^{1}x'), \ldots, f(T^{H_\e}x')) \in A \} &\ll \e^2.
\intertext{Unpacking definitions, we get}
\mu \{ x' \in X \colon  \sup_{\substack{ g \in G_i \\ x \in G_i / \Gamma_i }} | \E_{h \leq H_\e} f(T^{h}x') F(g^h x) | \geq \e \} &\ll \e^2.
\intertext{We call this set }
 \{ x' \in X \colon  \sup_{\substack{ g \in G_i \\ x \in G_i / \Gamma_i }} | \E_{h \leq H_\e} f(T^{h}x') F(g^h x) | \geq \e \} &= S_\e.
\end{align*}
Remember that $S_\e$ implicitly depends on $i$ and $F$. By the definition of the ergodic decomposition, we have that
\[
\mu(S_\e) = \int_\Omega \mu_\omega(S_\e) d\omega.
\]
Therefore, by another application of Chebyshev's inequality, we find that 
\[ | \{ \omega \in \Omega \colon \mu_\omega(S_\e) \leq \e  \} | \geq 1 - \e. \]
We call this set $K_{\e} = \{ \omega \in \Omega \colon \mu_\omega(S_\e) \leq \e  \}$. Of course $K_\e$ depends on $i$ and $F$. Define
\begin{align*}
\Omega_{i, F} &= \bigcap_{m \in \N}  \bigcup_{r \geq m} K_{\frac{1}{r}},
\intertext{
and define,
}
\Omega' &= \bigcap_{i \in \N} \bigcap_{F \in \funs} \Omega_{i, F}.
\end{align*}
Since $| K_\e | \geq 1- \e$, we know that for any $m$, we have $\left| \bigcup_{r \geq m} K_{\frac{1}{r}} \right| = 1$ and therefore $| \Omega' | =1$.

Now we check that $\Omega'$ has the desired properties. 
Thus, fix $\omega$ in $\Omega'$, $\phi$ a measurable function on $X$, $G$ a nilpotent Lie group of step $< \kappa$, $\Gamma$ a cocompact lattice and $F'$ a function on $G / \Gamma$. 
Suppose that for $\mu_\omega$ almost every $x$ in $X$, there exists $x'$ in $G / \Gamma$ such that $\phi(T^h x) = F'(g^h x')$ for some $g$ in $G$. Fix $\e > 0$. We aim to show
\[
\left| \int_X f(x) \phi(x) \mu_\omega(dx) \right| \lesssim \e  \cdot \left( || F' ||_{L^\infty} + 1 \right).
\]
Fix $i$ in the natural numbers such that $(G, \Gamma)$ is isomorphic to $(G_i, \Gamma_i)$. Fix $\psi \colon G_i \rightarrow G$ an isomorphism such that $\psi(\Gamma_i) = \Gamma$. Fix $F$ in $\funs_i$ such that $|| F \circ \psi - F' ||_{L^\infty} \leq \e$. Then $\omega$ is in $\Omega'$ so $\omega$ is in $\Omega_{i, F}$ and therefore there exists $r > \frac{1}{\e}$ such that $\omega$ is in $K_{\frac{1}{r}}$. Therefore, for some $H = H_{\frac{1}{r}}$,
\begin{align*}
\mu_\omega & \{ x' \in X \colon  \sup_{\substack{ g \in G_i \\ x \in G_i / \Gamma_i }} | \E_{h \leq H} f(T^{h}x') F(g^h x) | \geq \e \} \leq \e.
\intertext{
By the triangle inequality,
}
\mu_\omega & \{ x' \in X \colon  \sup_{\substack{ g \in G \\ x \in G / \Gamma }} | \E_{h \leq H} f(T^{h}x') F' (g^h x) | \geq 2 \e \} \leq \e.
\intertext{
Next, we use that $\phi$ locally looks like $F'$:
}
\mu_\omega & \{ x' \in X \colon  | \E_{h \leq H} f(T^{h}x') \phi(T^h x') | \geq 2 \e \} \leq \e.
\end{align*}
Bounding the exceptional points by the $L^\infty$ norm, we get that:
\begin{align*}
\int_X \left| \E_{h \leq H} f(T^h x) \phi(T^h x)  \right| \mu_\omega(dx) &\lesssim \e \cdot  \left( || F' ||_{L^\infty} + 1 \right).
\intertext{
By the triangle inequality,
}
\left| \int_X \E_{h \leq H} f(T^h x) \phi(T^h x)  \mu_\omega(dx) \right| &\lesssim \e \cdot \left( || F' ||_{L^\infty} + 1 \right).
\intertext{
By translation invariance,
}
\left| \int_X f(x) \phi(x) \mu_\omega(dx) \right| &\lesssim \e  \cdot   \left( || F' ||_{L^\infty} + 1 \right) .
\end{align*}
This completes the proof.
\end{proof}

\begin{prop}\label{HKexists}
Let $(X, \mu, T)$ be a (topologically) compact, invertible, not necessarily ergodic dynamical system. Let $\mu = \int_\Omega \mu_\omega d\omega$ be an ergodic decomposition. Recall that, for each $\omega$, the Host Kra factor $\z_\omega$ is defined up to sets of $\mu_\omega$-measure $0$. For each $\omega$, fix such a Host Kra factor. For instance, one could use any definition of the Host Kra factor and then add all sets of $\mu_\omega$-measure $0$ to obtain the complete Host Kra factor. Then there exists a sigma algebra $\z$ on $X$ such that, for any measurable set $A$, $A$ is $\z$ measurable if and only if there exists a full measure subset $\Omega' \subset \Omega$ such that for all $\omega$ in $\Omega'$, $A$ is $\z_\omega$ measurable. This implies that a function $f$ in $L^\infty(\mu)$ is $\z$ measurable if and only if there exists a full measure subset $\Omega' \subset \Omega$ such that $f$ is $\z_\omega$ measurable for every $\omega$ in $\Omega'$.
\end{prop}
\begin{proof}
Let $\z$ be the set of measurable subsets of $X$ such that there exists a full measure set $\Omega_A \subset \Omega$ such that for all $\omega$ in $\Omega_A$, $A$ is $\z_\omega$ measurable. For each such set, fix such an $\Omega_A$. Let $A_1, A_2, A_3, \ldots$ be a countable list of sets in $\z$. Consider 
\[
\Omega' = \bigcap_{i \in \N} \Omega_{A_i}.
\]
Because $\Omega'$ is the intersection of countably many full measure sets, it has full measure. Let $\omega$ be an element of $\Omega'$. Then for every natural number $i$, $A_i$ is $\z_\omega$ measurable. Because $\z_\omega$ is a sigma algebra, that implies the countable intersection and countable union of the sets $A_i$ are also $\z_\omega$ measurable. Thus the intersection $\cap A_i$ and union $\cup A_i$ are both $\z_\omega$ measurable for a full measure subset $\Omega' \subset \Omega$ and thus, by definition of $\z$, $\z$ is closed under countable unions and intersections. If $A$ is in $\z$, then $A$ is in $\z_\omega$ for every $\omega$ in $\Omega_A$. Since $\z_\omega$ is a sigma algebra, the complement $A^c$ is also in $\z_\omega$ for every $\omega$ in $\Omega_A$. By definition of $\z$, we conclude that $\z$ is closed under complements. Obviously $X$ and $\emptyset$ are in $\z$ so $\z$ is a sigma algebra.

Lastly, we check that a function $f$ in $L^\infty(\mu)$ is $\z$ measurable if and only if it is $\z_\omega$ measurable for a full measure set of $\omega$. First, suppose there exists a full measure subset $\Omega' \subset \Omega$ such that, for $\omega$ in $\Omega'$, $f$ is $\z_\omega$ measurable. Let $A$ be a measurable subset of $\C$. Then since $f$ is $\z_\omega$ measurable for any $\omega$ in $\Omega'$, $f^{-1}(A)$ is in $\z_\omega$ for any $\omega$ in $\Omega'$. Therefore, by definition of $\z$, $f^{-1}(A)$ is in $\z$ so $f$ is $\z$ measurable.

Now suppose $f$ is $\z$ measurable. We approximate $f$ by simple functions $f_i$. For instance, we can take $f_i(x) = k \cdot 2^{-i}$ if $f(x)$ is between $k \cdot 2^{-i}$ and $(k+1) \cdot 2^{-i}$ for any natural number $k$. Then $f_i \rightarrow f$ in $L^1(\mu)$ and also in $L^1(\mu_\omega)$ for any $\omega$ by the dominated convergence theorem. For each $i$, the function $f_i$ has only finitely many distinct level sets. Because $f$ is $\z$ measurable, the level sets of $f_i$ are $\z$ measurable. Therefore, there exists a full measure subset $\Omega_i \subset \Omega$ such that $f_i$ is $\z_\omega$ measurable for all $\omega$ in $\Omega_i$. Let
\[
\Omega' = \bigcap_{i \in \N} \Omega_i.
\]
Then since $\Omega'$ is the intersection of sets of full measure, $\Omega'$ has full measure. For each $\omega$ in $\Omega'$, $f_i$ is $\z_\omega$ measurable for all natural numbers $i$. But $f_i \rightarrow f$ in $L^1(\mu_\omega)$ so the limit $f$ is also $\z_\omega$ measurable for all $\omega$ in $\Omega'$. 
\end{proof}

\begin{defn}\label{HKdef}
By Proposition \ref{HKexists}, there exists a sigma algebra $\z$ such that a $L^\infty(\mu)$ function $f$ is $\z$ measurable if and only if it is $\z_\omega$ measurable for almost every $\omega$ in $\Omega$. We fix such a sigma algebra and call it the Host Kra sigma algebra for $(X, \mu, T)$.
\end{defn}


\begin{prop}\label{conditionwell}
Let $f$ be a function in $L^\infty(\mu)$. Then there exists a set $\Omega'$ of full measure in $\Omega$ such that for $\omega$ in $\Omega'$,
\[
\E^{\mu} [ f | \z] = \E^{\mu_\omega} [ f | \z_\omega],
\]
$\mu_\omega$ almost everywhere.
\end{prop}
\begin{proof}
First, we need the following quick ergodic theoretic fact.
The space $X$ can be essentially partitioned into pieces where each piece carries all the mass of an ergodic component. More precisely, there exists a map $\omega' \colon X \rightarrow \Omega$ such that 
\[
\int_\Omega \int_{X} \phi(x) \mu_\omega(dx) d\omega = \int_\Omega \int_{X} \phi(x) \ch_{\omega = \omega'(x)} \mu_\omega(dx) d\omega 
\]
for any integrable $\phi$. For instance, in the usual construction of an ergodic decomposition, one can take $\Omega$ to be the set of atoms of $X$ with respect to the invariant sigma algebra $\mathcal{I}$. Then let $\int \psi(x') \mu_{[x]}(dx') = \E [\psi | \mathcal{I} ](x)$ where $x$ is any point in the atom $[x]$. In this case the map $\omega'$ just sends $x$ to the atom containing $x$. 

By Proposition \ref{HKexists}, there is a set $\Omega_0$ of full measure such that 
$\E^{\mu}[ f | \z ]$ is $\z_\omega$ measurable for every $\omega$ in $\Omega_0$. We also ask that for $\omega$ in $\Omega_0$ that
\[
||f||_{L^\infty(\mu_\omega)} \leq || f ||_{L^\infty(\mu)}
\] 
which holds for a full measure set of $\omega$.
Fix such an $\Omega_0$. Since $X$ is compact, there exists a countable uniformly dense subset of the space of continuous functions. Fix such a subset and fix an order on that subset $f_1, f_2, f_3, f_4, \ldots$. Again by Proposition \ref{HKexists}, there exists a full measure subset $\Omega_i$ of $\Omega$ such that for $\omega$ in $\Omega_i$, the function $\E^{\mu}[ f_i | \z ]$ is $\z_\omega$ measurable. Let
\[
\Omega' = \bigcap_{i \geq 0} \Omega_i.
\]
Since each $\Omega_i$ has full measure and there are only countably many choices of $i$, we conclude that $\Omega'$ has full measure.
Now let $\omega$ be an element of $\Omega'$ and suppose for the sake of contradiction that
\[
\E^{\mu}[ f | \z ] \neq \E^{\mu_\omega}[ f | \z_\omega ],
\]
meaning equality does not hold up to sets of $\mu_\omega$ measure 0.
The conditional expectation is uniquely defined by two properties, namely that $\E^{\mu_\omega}[ f | \z_\omega ]$ is $\z_\omega$ measurable and that 
\[
\int_X \E^{\mu_\omega}[ f | \z_\omega ](x) \phi(x) \mu_\omega(dx) = \int_X f(x) \phi(x) \mu_\omega(dx),
\]
for any $\z_\omega$ measurable function $\phi$ in $L^\infty(\mu_\omega)$. If $\E^{\mu}[ f | \z ]$ satisfies the same properties then $\E^{\mu}[ f | \z ]$ equals $\E^{\mu_\omega}[ f | \z_\omega ]$ $\mu_\omega$-almost everywhere. We know since $\omega$ is in $\Omega'$ which is contained in $\Omega_0$ that $\E^{\mu}[ f | \z ]$ is $\z_\omega$ measurable. Therefore, there exists $\phi$ in $L^\infty(\mu_\omega)$ such that 
\[
\int_X \E^{\mu}[ f | \z ](x) \phi(x) \mu_\omega(dx) \neq \int_X f(x) \phi(x) \mu_\omega(dx).
\]
By subtracting off the appropriate multiple of $\E^{\mu}[ f | \z ]$, we may assume that $\phi$ is $\mu_\omega$-orthogonal to $\E^{\mu}[ f | \z ]$. Multiplying by a scalar we may assume that $\langle f, \phi \rangle_{L^2(\mu_\omega)}$ is a positive real number greater than 1.

For each $\omega$ in $\Omega'$ such that $\E^{\mu}[ f | \z ] \neq \E^{\mu_\omega}[ f | \z_\omega ]$, we showed there exists $\phi$ a $\z_\omega$ measurable function such that $\langle \E^{\mu}[ f | \z ], \phi \rangle_{L^2(\mu_\omega)} = 0$ and $\langle \phi, f \rangle_{L^2(\mu_\omega)} > 1$.
Let $\phi$ be such a function. Suppose for the moment that $|| \phi ||_{L^2(\mu_\omega)} < C$.
Since $f_1, f_2, f_3, \ldots$ are dense in $L^2(\mu_\omega)$, for any $\e$ and for any power $p \geq 2$ we can find an $i$ such that $|| \phi - f_i ||_{L^p(\mu_\omega)} \leq \e$. This implies, by Cauchy-Schwarz, that
\begin{equation}\label{e1}
\langle \E^{\mu}[ f | \z ], f_i \rangle_{L^2(\mu_\omega)} \leq \e ||f ||_{L^\infty(\mu)}
\end{equation}
and
\begin{equation}\label{e2}
\langle \phi, f_i \rangle_{L^2(\mu_\omega)} > 1 - \e.
\end{equation}
We also need a quantitative way of saying that $f_i$ is close to being $\z_\omega$ measurable. One option is to use the Host Kra norms
defined for an ergodic system in \cite{HostKra} section 3.5.
 Let
$\| \phi \|_{k, \omega}$ denote the $k^{th}$ Host Kra norm.
The key feature of the Host Kra norms is that a function $\phi$ is $\z_\omega$ measurable if and only if $\| \phi \|_{k, \omega} = 0$ for all $k$ (see \cite{HostKra} Lemma 4.3). We claim that $\| \phi \|_{k, \omega}$ is a measurable function of $\omega$. After all, by definition $ \| \phi \|_{k, \omega}^{2^k} $ is the integral of some fixed function on $X^{2^k}$, namely $(x_1, \ldots, x_{2^k}) \mapsto  \phi(x_1) \ldots \phi(x_{2^k})$ with respect to some measure (namely $\mu_{\omega}^{[k]}$ defined in \cite{HostKra} section 3.1) which depends measurably on $\omega$. Thus, we can find also $f_i$ with
\begin{equation}\label{e4}
\| f_i \|_{k, \omega} \leq 2 \e,
\end{equation}
for all $k \leq \frac{1}{\e}$.
If we also know that
\[
|| \phi ||_{L^p(\mu_\omega)} < C
\]
for some constant $C$ then also by the triangle inequality,
\begin{equation}\label{e3}
|| f_i ||_{L^p(\mu_\omega)} \leq C + \e.
\end{equation}

Fix a constant $C$.
Now we define a function $i \colon \Omega \times \R_{> 0} \rightarrow \N \cup \{ \infty \}$  as follows:
Let $i(\omega, \e)$ be the first index such that all four inequalities \ref{e1} - \ref{e3} are satisfied with $p = \frac{1}{\e}$ if such an $i$ exists and $+\infty$ otherwise. Note that $i$ implicitly depends on $C$. Let $E$ denote the set of $\omega$ such that $i(\omega, \e)$ is finite for all $\e$. In particular, if $\E^{\mu}[ f | \z ] \neq \E^{\mu_\omega}[ f | \z_\omega ]$ then $\omega$ is in this set for some choice of $C$. Thus, we may assume for the sake of contradiction that the measure of $E$ is positive.
 Let
\[
\psi_\e(x) = 
\begin{cases}
 f_{i(\omega, \e)}(x)  & \omega \in E, \omega'(x) = \omega \\
 0 & \text{ otherwise}
\end{cases}
\]
Since, for all $p < \frac{1}{\e}$
\[
|| \psi_\e ||_{L^p(\mu)} = \int_\Omega || \psi_\e ||_{L^p(\mu_\omega)} d\omega \leq C + \e,
\]
we can take an $L^p(\mu)$ weak-$*$ limit $\psi_\e \rightarrow \psi$ for some subsequence of epsilons tending to $0$. By a diagonalization argument, we can ensure that this weak-$*$ limit exists for all $p < \infty$. 
By \ref{e4}, we conclude that $\psi$ is $\z_\omega$ measurable for each $\omega$ in $E$. If $\omega$ is not in $E$, then $\psi = 0$ on a set of $\mu_\omega$ full measure so $\psi$ is measurable with respect to $\z_\omega$ for a full measure set of $\omega$ in $\Omega$ so by definition $\psi$ is $\z$ measurable. Futhermore, by \ref{e2}
\[
\langle \phi, \psi \rangle_{L^2(\mu_\omega)} \geq 1
\]
so we conclude that
\[
\langle \phi, \psi \rangle_{L^2(\mu)} \geq |E|
\]
by integrating in $\omega$. On the other hand, by \ref{e1}
\[
\langle \E^{\mu}[ f | \z ], \psi \rangle_{L^2(\mu_\omega)} = 0.
\]
This contradicts the definition of $\E^{\mu}[ f | \z ]$. Thus,
\[
\E^{\mu} [ f | \z] = \E^{\mu_\omega} [ f | \z_\omega]
\]
for almost every $\omega$ in $\Omega$.
\end{proof}


	A crucial input is the following theorem of \cite{FrantzikinakisHost2}. This theorem says that if $a$ correlates with $b$ then it does so for some algebraic reason. In particular, any correlation between $f$ and $f'$ is due solely to some locally algebraic structure in $f$.
	\begin{thm}[\cite{FrantzikinakisHost2} Theorem 1.5; see also the first appendix to this paper]\label{fh2}
		Let $\mu$ be the first marginal of $\nu$ corresponding to the factor $X$. Then the ergodic components $(X, \mu_\omega, T)$ of $\mu$ are isomorphic to the product of a Bernoulli system with the Host Kra factor of $(X, \mu_\omega, T$).
	\end{thm}
	To use this theorem, we need the following result, which essentially appears in \cite{FrantzikinakisHost}:
	\begin{lem}[\cite{FrantzikinakisHost}; see the proof of Lemma 6.2]\label{disjoint}
		Suppose that $(X, \mu_\omega,T) \cong (W, dw, T) \times (Z, dz,T)$ where $W$ is a Bernoulli system, $Z$ is a zero entropy system and $\mu_\omega$ is the first marginal of $\nu_\omega$. Then for any function $\phi \colon X \rightarrow \C$ and any function $\psi \colon Y \rightarrow \C$ we have
		\[
		\int_{X \times Y} \phi(x) \psi(y) \nu_\omega(dx dy) = \int_{X \times Y} \E^{\nu_\omega}[\phi | Z](x) \psi(y) \nu_\omega(dx dy)
		\]
where $\E^{\nu_\omega}[ \phi | Z]$ denotes the conditional expectation of $\phi$ with respect to the measure $\nu_\omega$ and the sigma algebra of $Z$-measurable functions.
	\end{lem}
	\begin{proof}
		By density, it suffices to consider the case $\phi(w,z) = \phi_W(w) \phi_Z(z)$. Because any joining of the Bernoulli system $W$ and the zero entropy system $Z \times Y$ is trivial i.e. is equipped with the product measure, we can break up the the integral
		\begin{align*}
		\int_{W \times Z \times Y} \phi_W(w) \phi_Z(z) \psi(y) \nu_\omega(dw dz dy) =& \int_W \phi_W(w) \nu_\omega(dw dz dy)  \cdot \int_{Z \times Y} \phi_Z(z) \psi(y) \nu_\omega(dw dz dy) \\
		=& \int_{X \times Y} \E^{\nu_\omega}[ \phi | Z ](z) \psi(y) \nu_\omega(dw dz dy).
		\end{align*}
	\end{proof}
	We also need the following result, which says that conditional expectation is essentially local.

	Putting everything together gives the following corollary.
	\begin{cor}\label{startingpoint}
	Let $X$, $Y$, $\nu$, $f$, $f'$, $c$ and $\eta$ be as in Subsection \ref{background}. Let $B$ be as in Corollary \ref{periodicpoints}. Then
		\[
		\left| \int_{B^c} \E_{h \leq k} T^h (\E^{\nu} [f | \z] \cdot f') d\nu \right| > c + \eta.
		\]
	\end{cor}
	\begin{proof}
		Recall that
		\begin{align*}
		&\left| \int_{X \times Y} f(x) \cdot f'(y) \nu(dx dy) \right| > c + \eta.
		\intertext{By Corollary \ref{periodicpoints}, we have that}
		&\left| \int_{B^c} f(x) \cdot f'(y) \nu(dx dy) \right| > c + \eta.
		\intertext{Since $B^c$ is $T$ invariant and $\nu$ is $T$ invariant, we can average over shifts}
		&\left| \int_{B^c} \E_{h \leq k} f(T^h x) \cdot f'(T^h y) \nu(dx dy) \right| > c + \eta.
		\intertext{Next, we disintegrate the measure $\nu$,}
		\bigg| \int_\Omega &\int_{B^c} \E_{h \leq k} f(T^h x) \cdot f'(T^h y) \nu_\omega(dx dy) d\omega \bigg| > c + \eta.
		\intertext{Notice, for each $h$, $f'(T^h y) \ch_{y \not\in B}$ is a function on $Y$. By Theorem \ref{fh2}, $(X, \mu_\omega, T)$ is isomorphic to a product of a Bernoulli factor with the Host Kra factor for almost every $\omega$. Since the Host Kra factor has entropy zero, the hypotheses of Lemma \ref{disjoint} are satsified. Thus, by Lemma \ref{disjoint}, }
		\bigg| \int_\Omega &\int_{B^c} \E_{h \leq k} \E^{\mu_\omega} [ f | \z_\omega ](T^h x) \cdot f'(T^h y) \nu_\omega(dx dy) d\omega \bigg| > c + \eta.
		\intertext{By Proposition \ref{conditionwell}, }
		\bigg| \int_\Omega &\int_{B^c} \E_{h \leq k} \E^{\mu} [ f | \z ](T^h x) \cdot f'(T^h y) \nu_\omega(dx dy) d\omega \bigg| > c + \eta.
\intertext{
By definition of the ergodic decomposition,
}
\bigg| &\int_{B^c} \E_{h \leq k} \E^{\mu} [ f | \z ](T^h x) \cdot f'(T^h y) \nu(dx dy) \bigg| > c + \eta.
		\end{align*}
		This completes the proof.
	\end{proof}
	Now we forget everything about the joining of $X$ and $Y$ and reduce to the worst case scenario, where we choose the worst possible $y$ in $Y$ for each $x$ in $X$. 
	\begin{cor}\label{startingpoint2}
	Let $X$, $Y$, $\nu$, $\mu$, $f$, $f'$, $c$ and $\eta$ be as in Subsection \ref{background}. Let $\z$ be as in Definition \ref{HKdef}. Let $B$ be as in Corollary \ref{periodicpoints}. 
		Since whether $(x,y) \in B$ only depends on $y$, we abuse notation and write $y \in B$ to mean $(x,y) \in B$ for some $x$. Then,
		\[
		\int_X  \sup_{y \not\in B} \Big| \E_{h \leq k} \E^{\mu} [f | \z] (T^h x) \cdot f'(T^h y) \Big|  \mu(dx) > c + \eta,
		\]
		where the supremum is an essential supremum taken with respect to the second marginal of $\nu$.
	\end{cor}

	We will need the following lemma, which states that conditioning with respect to a conditional measure is essentially the same as conditioning with respect to the original measure.
	\begin{lem}\label{relcondexp}
		Let $A$ be a positive measure set in $\z$ and denote $\mu_A (S) = \mu(S | A)$. Then for any measurable function $f$, 
		\[
		\E^{\mu_A}[f | \z] = \E^{\mu}[f | \z],
		\]
		$\mu_A$ almost everywhere i.e. for $\mu$-almost every point in $A$.
	\end{lem}
	\begin{proof}
		Let $C$ be another set in $\z$. Then 
		\begin{align*}
		\int_{C} \E^{\mu}[ f | \z ](x) \mu_A(dx) =& \int_{C \cap A} \frac{1}{\mu(A)} \E^{\mu}[ f | \z ](x) \mu(dx) 
		\intertext{Since $A$ is in $\z$, we know that $A \cap C$ is in $\z$. By definition of conditional expectation, this is}
		=& \frac{1}{\mu(A)} \int_{C \cap A} f \mu(dx) \\
		=& \int_{C \cap A} f \mu_A(dx).
		\end{align*}
		This is the defining property of $\E^{\mu_A} [f | \z ]$. Since conditional expectation is well defined up to sets of measure $0$, we obtain the result.
	\end{proof}
	The system $X$ possesses an extra symmetry that most dynamical systems do not have, a dilation symmetry. In fact, it possesses a whole family of dilation symmetries. It is not obvious which dilation makes the problem easiest. Therefore, instead of choosing a particular dilation, we use a random dilation. 
	\begin{prop}
		Let $P$ be any natural number. Then
		\[
		\E_{P/2 < p \leq P} \int_X   p \ch_{M^{-1}(p \hat{\Z} )}(x) \sup_{y \not\in B} \Big| \E_{h \leq k} \E^{\mu} [f | \z] (T^{ph} x) \cdot f'(T^h y) \Big|  \mu(dx) > c + \eta,
		\]
		where $p$ is always restricted to be prime.
	\end{prop}
	\begin{proof} 
		By Corollary \ref{startingpoint2} we have 
		\[
		\int_X  \sup_{y \not\in B} \Big| \E_{h \leq k} \E^{\mu} [f | \z] (T^h x) \cdot f'(T^h y) \Big|  \mu(dx) > c + \eta.
		\]
		Now we use that $I_{p}$ pushes forward $p \ch_{M^{-1}(p \hat{\Z} )} \mu$ to $\mu$ for every $p$ and average in $p$.
		\[
		\E_{P/2 < p \leq P} \int_{X} p \ch_{M^{-1}(p \hat{\Z} )} \sup_{y \not\in B} \Big| \E_{h \leq k} \E^{\mu} [f | \z] (T^h I_p x) \cdot f'(T^h y) \Big|  \mu(dx) > c + \eta.
		\]
		Because $I_p \circ T^{hp}(x) = T^h \circ I_p (x)$ for almost every $x$ in $M^{-1}(p \hat{\Z} )$ we have that,
		\[
		\E_{P/2 < p \leq P} \int_{X} p \ch_{M^{-1}(p \hat{\Z} )} \sup_{y \not\in B} \Big| \E_{h \leq k} \E^{\mu} [f | \z] (I_p T^{ph} x) \cdot f'(T^h y) \Big|  \mu(dx) > c + \eta.
		\]
		Next we use the standard fact that 
		\[
		\E^{\mu}[ f | \z ] \circ I_p = \E^{I_{p*}  \mu}[ f \circ I_p | I_p^{-1}(\z) ],
		\]
		$I_{p*} \mu$-almost everywhere, where $I_{p*} \mu$ is the pushforward of $\mu$. Since $I_{p*} \mu = \frac{1}{p} \mu$, we can replace $I_{p*} \mu$ by $\mu$. 
Note that $I_p$ defines a factor map between $(M^{-1}(p \widehat{\Z}), p \mu, T^p)$ and $(X, \mu, T)$. Since Host Kra factors are functorial, the Host Kra factor for $(M^{-1}(p \widehat{\Z}, p \mu, T^p) )$ factors onto the Host Kra factor for $(X, \mu, T)$. Thus, $I^{-1}_p(\z)$ is contained in the Host Kra factor of some dynamical system and thus corresponds to an inverse limit of nilsystems. This is all we actually need for our purposes. However, for the sake of avoiding notation, we also prove that 
\[
\ch_{M^{-1}(p \widehat{\Z}) } \E^{\mu}[ f \circ I_p | I_p^{-1}(\z) ] = \ch_{M^{-1}(p \widehat{\Z}) } \overline{a(p)} \E^{\mu}[ f | \z ].
\]
That $f \circ I_p = \overline{a(p)} f$ follows from the definition of $I_p$. If $\z_i(T^p)$ denotes the $i^{th}$ Host Kra factor for $T^p$ and $\z_i(T)$ denotes the $i^{th}$ Host Kra factor for $T$, then any $T^p$ invariant subset of the cube $X^{2^i}$ is an element of the Konecker factor i.e. the first Host Kra factor for $(X^{2^i}, T, \mu^{[i]})$ (where $\mu^{[i]}$ is the measure on the cube defined in section 3 of \cite{HostKra}). Since the Host Kra factor of an ergodic system is the smallest sigma algebra generating the invariant factor on the cube, we conclude that $\z_i(T^p) \subset \z_{i+1}(T)$ so $I^{-1}_p(\z) \subset \z \cap M^{-1}(p \widehat{\Z})$. In fact, as in the appendix, the Host Kra factor for $X$ is a joining of the Host Kra factor on the space of sequences $D^\Z$ and $\widehat{\Z}$. On the second factor, $I_p$ acts by division by $p$. On the first factor, $I_p \circ T^p = T \circ I_p$ and so on each ergodic component of the first factor, $I_p$ acts by multiplication by $p$ up to a possible translation. Multiplication by $p$ is a local isomorphism of any nilmanifold that does not contain $p$ torsion. However, by Corollary \ref{MRTerg}, $f$ is already orthogonal to all $p$ torsion. Thus, 
\[
\ch_{M^{-1}(p \widehat{\Z}) } \E^{\mu}[ f \circ I_p | I_p^{-1}(\z) ] = \ch_{M^{-1}(p \widehat{\Z}) } \overline{a(p)} \E^{\mu}[ f | \z ].
\]
 Combined with Lemma \ref{relcondexp} and the fact that $M^{-1}(p \hat{\Z} )$ is $T^p$ invariant and therefore an element of $\z$ we get,
		\[
		\E_{P/2 < p \leq P} \int_{X} p \ch_{M^{-1}(p \hat{\Z} )}  \sup_{y \not\in B} \Big| \E_{h \leq k} \E^{\mu} \overline{a(p)} [f | \z] (T^{ph} x) \cdot f'(T^h y) \Big|  \mu(dx) > c + \eta.
		\]
		Recall that $|a(p)| = 1$ for all $p$. Thus, $\overline{a(p)}$ merely gets absorbed into the absolute value.
		\[
		\E_{P/2 < p \leq P} \int_{X} p \ch_{M^{-1}(p \hat{\Z} )}  \sup_{y \not\in B} \Big| \E_{h \leq k} \E^{\mu} [f | \z] ( T^{ph} x) \cdot f'(T^h y) \Big|  \mu(dx) > c + \eta.
		\]
	\end{proof}
\subsection{The Entropy Decrement Argument}
	Next, we use the entropy decrement method to replace $p \ch_{M^{-1}(\hat{\Z})}$ by its average, $1$. This is essentially due to Tao but because our statement is slightly different we reproduce the argument. For the definitions of entropy, conditional entropy, mutual information and conditional mutual information see \cite{TaoChowla}.

Let $x'$ be a random variable distributed according to $\mu$ and fix a natural number $P$. From this, we get the following two random variables. Set $X_P = (x_1, \ldots, x_{(k+1)P})$ where $x_i = \E^\mu [f | \z] (T^i x')$ and set $Y_P$ in $\prod_{P/2 < p \leq P} \Z / p \Z$ by $Y_P = (M(x') \mod p)_{P/2 < p \leq P}$. Denote $Y_P \mod p = y_p$ so that $Y_P = (y_p)_{P/2 < p \leq P}$. Note that $Y_P$ is uniformly distributed in $\prod_{P/2 < p \leq P} \Z / p \Z$ and that the distribution of $X_P$ is the same as the distribution of $T^i X_P$ for any $i$ because $\mu$ is translation invariant. Technically, if $\E^\mu[f | \z]$ takes infinitely many values then we will have to round $\E^\mu[f | \z](T^i x')$ so that each $x_i$ takes values in a finite set but this slightly annoying detail may be delayed for the moment.
We want to study the following integral:
		\[
		\E_{P/2 < p \leq P} \int_{X} p \ch_{M^{-1}(p \hat{\Z} )}  \sup_{y \not\in B} \Big| \E_{h \leq k} \E^{\mu} [f | \z] ( T^{ph} x) \cdot f'(T^h y) \Big|  \mu(dx) 
		\]
By translation invariance, this is equal to 
		\[
		\E_{P/2 < p \leq P} \E_{i \leq P} \int_{X} p T^i \ch_{M^{-1}(p \hat{\Z} )}  \sup_{y \not\in B} \Big| \E_{h \leq k} \E^{\mu} [f | \z] ( T^{ph+i} x) \cdot f'(T^h y) \Big|  \mu(dx).
		\]
Notice that this is the expected value of some function of $X_P$ and $Y_P$. In particular, we are interested in 
\[
\E \left( \E_{P/2 < p \leq P}  f_{P,p} (X_P, y_p) \right)
\]
where $f_{P, p} \colon \C^{(k+1)P} \times \Z / p \Z \rightarrow \C$ is defined by the formula
\[
f_{P,p} (X_P, y_p) = \E_{i \leq P} \ p \ch_{y_p = i} \sup_{y \not\in B} \Big| \E_{h \leq k} x_{hp + i} f(T y) \Big|.
\]
Define
\[
f_P(X_P, Y_P) = \E_{P/2 < p \leq P} f_{P, p}(X_P y_P). 
\]
Thus, we are interested in 
\[
\E f_P(X_P, Y_P).
\]
We would like to say that $X_P$ and $Y_P$ are very close to independent for some large choice of $P$. Let $W_P$ be a random variable with the same distribution as $Y_P$ but which is independent of $X_P$. We would like to say that
\[
\E[ f_P(X_P, Y_P) ] \approx \E[ f_P(X_P, W_P) ].
\]
A property like this actually holds in a more general setting, which we take the liberty of stating now.
\begin{thm}\label{tao}[\cite{TaoChowla} Section 3; see also \cite{Blog2},\cite{TJ} Lemma 3.4 and Proposition 3.5 and \cite{TJ2} Section 4]
Let $A$ be a finite set and let $C$ be a natural number. For each power of two $P$, let $X_P = (x_1, \ldots, x_{CP})$ be a sequence of random variables with $x_i$ taking values in $A$ and let $Y_P$ be a random variable that is uniformly distributed in $\prod_{P/2 < p \leq P} \Z / p \Z$. We write $Y_P = (y_p)_{P/2 < p \leq P}$ where $y_p = Y_P \mod p$. We further assume that for different values of $P$, the random variables $Y_P$ are jointly independent meaning $(y_p)_{p \leq P}$ is uniformly distributed in $\prod_{p \leq P} \Z / p \Z$ for all powers of two $P$. Suppose that, for any natural numbers $i$ and $m$ such that $i + m \leq CP$ we have that the distribution of $(x_1, \ldots, x_m)$ is equal to the distribution of $(x_{i+1}, \ldots, x_{i + m})$. Furthermore, suppose that for any $P$ and any element $b$ in $\prod_{p \leq P} \Z / p \Z$ and any $S$ a measurable subset of $\C^{m}$, \[ \mathbb{P}( (x_1,\ldots, x_m) \in S \ | \ (y_p)_{p \leq P} = b) = \mathbb{P}( (x_{i +1},\ldots, x_{i+m}) \in S \ | \ (y_p)_{p \leq P} = b + i). \] For each $p$ with $P/2 < p \leq P$, let $f_{P,p}$ be a $1$-bounded function $A^{CP} \times \Z / p \Z \rightarrow \C$ and let $f_P(X_P, Y_P) = \E_{P/2 < p \leq P} f_{P, p}(X_P, y_p)$. Let $W_P$ be a random variable with the same distribution as $Y_P$ but which is independent of $X_P$. Then
\[
\liminf_{P \rightarrow \infty} \E [ | f_P(X_P, Y_P) - f_P(X_P, W_P) | ] = 0.
\]
\end{thm}
	\begin{proof}
Fix a large power of two $P$ and $\e > 0$.
By replacing $f_{P,p}(a,b)$ by $f_{P,p}(a,b) -  f_{P, p}(a, W_P)$ we may assume that $ f_{P, p}(a, W_P) = 0$ for all $a$. 
		To prove the theorem, first we need a very good understanding of the case when $X_P$ and $Y_P$ are independent. In that case, even if we know the exact value of $X_P$, $f_P$ is still a sum of independent random variables $f_{P, p}(a, W_P)$ and therefore exhibits concentration. This is formalized in Hoeffding's inequality, which says that large collections of independent random variables exhibit concentration.
		\begin{lem}[Hoeffding's Inequality]
			Suppose $Z_1, \ldots, Z_n$ are independent random variables taking values in $[-2, 2]$. Then 
			\[ \mathbb{P}( |Z_1 + \cdots + Z_n - \E[Z_1 + \cdots + Z_n] | > t ) \leq \exp ( - n t^2 / 16  ). \]
		\end{lem}
Let $a$ be an element of $A^{CP}$. We apply Hoeffding's inequality to the random variables $f_{P,p}(a, Y_P)$. (We remind the reader that there are roughly $\frac{P}{2 \log P}$ many such terms, by the prime number theorem).
\begin{equation}\label{hoeffcor}
\mathbb{P} ( |f_P(a, W_P)| > \e ) \leq  \exp(-\frac{\e^2 P}{40 \log P}).
\end{equation}
		Next, we aim to show that if $Y_P$ is not necessarily independent of $X_P$ but nearly independent of $Y_P$, we still can obtain a good bound. 
To do this, we use a Pinsker-type inequality.
		\begin{lem}\label{pinsker}[\cite{TJ2} Lemma 3.4]
			Let $Y$ be a random variable taking values in a finite set, let $W$ be a uniformly distributed random variable on the same set and let $E$ be a set. Then
			\[\mathbb{P}(Y \in E) \leq -\frac{H(W) - H(Y) + \log 2}{ \log \mathbb{P}(W \in E) }. \]
		\end{lem}
Let $a$ be an element of $A^{CP}$. Let $E$ be the set of $b$ in $\prod_{P/2 < p \leq P} \Z / p \Z$ such that $| f_P(a,b) | > \e$.  By \ref{hoeffcor}, we know 
\[
\mathbb{P}(W_P \in E)  \leq  \exp(-\frac{\e^2  P}{40 \log P}).
\]
Applying Lemma \ref{pinsker} to $\mathbb{P}( \ \cdot \ | X_P = a)$, we find 
\[\mathbb{P}( |f(a,Y_P)| > \e | X_P = a ) \leq -\frac{(H(W_P) - H(Y_P | X_P = a) + \log 2) 40 \log P}{ \e^2  P}. \]
Note that 
\[
\sum_a \mathbb{P}(X_P = a) (H(W_P) - H(Y_P | X_P = a)) = H(W_P) - H(Y_P | X_P) = I(X_P, Y_P),
\]
where the last equality follows since $H(W_P) = H(Y_P)$ since the two random variables have the same distribution. Therefore, summing over $a$, we get
\[
\mathbb{P}( |f(X_P,Y_P)| > \e ) \leq -\frac{(I(X_P, Y_P) + \log 2) 40 \log P}{ \e^2  P}.
\]
If
\[
I(X_P, Y_P) \lesssim \frac{\e P}{ \log P}
\]
then 
\[
\mathbb{P}( |f(X_P, Y_P)| > \e) \lesssim \e.
\]
This would complete the proof. Let $Y_{\leq P/2} = (y_p)_{p \leq P/2}$.  Fix $b'$ an element of $\prod_{p \leq P/2} \Z / p \Z$. Then we may repeat the previous argument with $\mathbb{P} ( \ \cdot \ | Y_{\leq P/2} = b')$ to conclude:
\[
\mathbb{P}( |f(X_P,Y_P)| > \e ) \leq -\frac{(I(X_P, Y_P | Y_{\leq P/2 } ) + \log 2) 40 \log P}{ \e^2  P}.
\]
and therefore if
\[
I(X_P, Y_P | Y_{\leq P/2 } )  \lesssim \frac{\e^3  P}{ \log P}
\]
then 
\[
\mathbb{P}( |f(X_P,Y_P)| > \e ) \lesssim \e
\]
and therefore
\[
\E |f(X_P,Y_P)| \lesssim \e.
\]
Let $P_0$ be a power of two. We will try to show that there exists $P \geq P_0$ such that 
\[
\E |f(X_P,Y_P)| \lesssim \e.
\]
This would complete the proof.
Suppose not.
Then 
\[
I(X_P, Y_P | Y_{\leq P/2 } )  \gg \frac{\e^3  P}{\log P},
\]
for all $P \geq P_0$. By definition of mutual information,
\begin{align}
H(X_P | Y_{\leq P}) &= H(X_P | Y_{\leq P/2} ) - I(X_P, Y_P | Y_{\leq P/2}) \nonumber
\intertext{where $Y_{\leq P} = (y_p)_{p \leq P}$. By assumption, we have a lower bound for the mutual information}
&\leq H(X_P | Y_{\leq P/2} ) - \frac{\e^3  P}{ \log P} \nonumber
\intertext{By subadditivity of entropy, }
&\leq H(X_{P/2} | Y_{\leq P/2} ) + H(x_{CP/2 + 1}, \ldots, x_{CP} | Y_{\leq P/2} ) - \frac{\e^3  P}{ \log P} \label{entropyformula}
\end{align}
where $X_{P/2} = (x_1, \ldots, x_{CP/2})$. Since $(x_1, \ldots, x_{CP/2})$ has the same distribution as $(x_1, \ldots, x_{CP/2})$, for any set $S$ in $\C^{CP/2}$ and for any $b'$ in $\prod_{p \leq P/2} \Z / p \Z$
\[
\mathbb{P}( (x_1, \ldots, x_{CP/2}) \in S | Y_{\leq P/2} = b' )  = \mathbb{P}( (x_{CP/2 + 1}, \ldots, x_{CP}) \in S | Y_{\leq P/2} = b' +  CP/2 ).
\]
Since the entropy of a random variable only depends on its distribution, we conclude that, for all $b'$
\[
H( x_1, \ldots, x_{CP/2}  | Y_{\leq P/2} = b' ) = H( x_{CP/2 + 1}, \ldots, x_{CP} | Y_{\leq P/2} = b' +  CP/2 )
\]
Since $Y_{\leq P/2}$ is uniformly distributed, for all $b'$,
\[
\mathbb{P}(Y_{P/2} = b') = \mathbb{P}(Y_{P/2} = b' + CP/2).
\]
Therefore, summing in $b'$,
\begin{align}
& H( x_1, \ldots, x_{CP/2}  | Y_{\leq P/2}) \label{entropyformula2} \\
= &\sum_{b'} \mathbb{P}(Y_{P/2} = b') H( x_1, \ldots, x_{CP/2}  | Y_{\leq P/2} = b' ) \nonumber
\\
= & \sum_{b'}\mathbb{P}(Y_{P/2} = b' + CP/2) H( x_{CP/2 + 1}, \ldots, x_{CP} | Y_{\leq P/2} = b' +  CP/2 ) \nonumber \\
= & H( x_{CP/2 + 1}, \ldots, x_{CP} | Y_{\leq P/2} ). \nonumber
\end{align}
Applying \ref{entropyformula2} to \ref{entropyformula},
\begin{align*}
H(X_P | Y_{\leq P}) &\leq 2H(X_{P/2} | Y_{\leq P/2} )  - \frac{\e^3  P}{ \log P}.
\intertext{We just obtained an upper bound for $H(X_P | Y_{\leq P})$. We can apply the same argument to obtain an upper bound for $H(X_{P/2} | Y_{\leq P/2} )$. }
&\leq 4H(X_{P/4} | Y_{\leq P/4} ) - 2 \frac{\e  P/2}{\log  P/2}  - \frac{\e^3  P}{\log P}
\end{align*}
where $X_{P/2} = (x_1, \ldots, x_{CP/4})$ and where $Y_{\leq P /4} = (y_p)_{p \leq P /4}$. Applying this argument inductively, if $P = 2^m \cdot P_0$ then
\begin{equation}\label{entropyformula3}
\leq 2^m \left( H(X_{P_0} | Y_{\leq P_0}) - \e \sum_{j = 1}^{m} \frac{P_0 }{j} \right).
\end{equation}
However, $\sum_{m \leq \log_2 P / P_0 }\frac{1}{m} \sim \log \log P$ so for large $P$
\[
\e \sum_{j = 1}^{m} \frac{P_0 }{j} \gg C P_0 \log |A| \geq H(X_{P_0} | Y_{\leq P_0}).
\]
Combining this with \ref{entropyformula},
\[
H(X_P, Y_{\leq P}) < 0
\]
which is impossible.

	\end{proof}
	Applying the Theorem \ref{tao} to our situation yields,
	\begin{cor}\label{cor1}
	Let $X$, $Y$, $\mu$, $f$, $f'$, $M$, $I_p$, $c$ and $\eta$ be as in Subsection \ref{background}. Let $\z$ be as in Definition \ref{HKdef}. Let $B$ be as in Corollary \ref{periodicpoints}. 
		We have
		\[
		\limsup_{P \rightarrow \infty} \E_{P/2 < p \leq P} \int_X \sup_{y \not\in B} \Big| \E_{h \leq k} \E^{\mu} [f | \z] (T^{ph} x) \cdot f'(T^h y) \Big|  \mu(dx) > c.
		\]
	\end{cor}
	\begin{proof}
Recall that, for all natural numbers $P$,
		\[
		\E_{P/2 < p \leq P} \int_{X} p \ch_{M^{-1}(p \hat{\Z} )}  \sup_{y \not\in B} \Big| \E_{h \leq k} \E^{\mu} [f | \z] ( T^{ph} x) \cdot f'(T^h y) \Big|  \mu(dx) > c + \eta.
		\]
By translation invariance, for all natural numbers $P$,
		\[
		\E_{P/2 < p \leq P} \E_{i \leq P} \int_{X} p \ch_{M^{-1}(p \hat{\Z} + i )}  \sup_{y \not\in B} \Big| \E_{h \leq k} \E^{\mu} [f | \z] ( T^{ph + i} x) \cdot f'(T^h y) \Big|  \mu(dx) > c + \eta.
		\]
Let $x'$ be a random variable with distribution $\mu$. Fix $\e > 0$ small. We will ask that $\e < 10 \cdot \eta$. Let $\phi$ be a measurable function on $X$ which uniformly approximates $ \E^{\mu} [f | \z]$ i.e.
\[
|| \phi - \E^{\mu} [f | \z] ||_{L^\infty} < \e.
\]
For instance, $\phi(x)$ could be obtained by rounding $\E^{\mu} [f | \z](x)$ to the closest element of $\frac{\e}{10} \cdot \Z[i]$. By the triangle inequality
		\[
		\E_{P/2 < p \leq P} \E_{i \leq P} \int_{X} p \ch_{M^{-1}(p \hat{\Z} + i )}  \sup_{y \not\in B} \Big| \E_{h \leq k} \E^{\mu} \phi ( T^{ph + i} x) \cdot f'(T^h y) \Big|  \mu(dx) > c + \eta - \e.
		\]
For each natural number $P$, let $X_P = (x_1, \ldots, x_{CP})$ where $x_i = \phi(T^i x')$ and where $C = k+1$. Let $Y_P = (y_p)_{P/2 < p \leq P}$ where $y_p = M(x') \mod p$. For each natural number $P$, let 
\[
f_{P,p} (X_P, y_p) = \E_{i \leq P} \ p \ch_{y_p = i} \sup_{y \not\in B} \Big| \E_{h \leq k} x_{hp + i} f(T y) \Big|.
\]
Define
\[
f_P(X_P, Y_P) = \E_{P/2 < p \leq P} f_{P, p}(X_P y_P). 
\]
Unpacking definitions, for every natural number $P$,
\[
\E f_P(X_P, Y_P) > c + \eta - \e.
\]
Now we check the hypotheses of Theorem \ref{tao}. Because $\phi$ takes only finitely many values, $x_i$ takes values in a finite set.  For all natural numbers $P$, since the distribution of $(y_p)_{p \leq P}$ is a $+1$ invariant measure on $\prod_{p \leq P} \Z /p \Z$, it must be the uniform distribution. Since $\mu$ is translation invariant, for any natural numbers $i$ and $m$ and any subset $E$ of $X^m$
\[
\mathbb{P} ( (T x', \ldots, T^m x') \in E ) = \mathbb{P} ( (T^{i+1} x', \ldots, T^{i + m} x') \in E ).
\]
Applying this to the preimage under $(\phi,\ldots,\phi)$ of an arbitrary subset $S$ of $\C^m$ reveals that the distribution of $(x_1, \ldots, x_m)$ is the same as the distribution of $(x_{i+1}, \ldots, x_{i+m})$. Similarly, if $b$ is an element in $\prod_{p \leq P} \Z / p \Z$ if $E$ is the preimage under $(\phi, \ldots, \phi)$ of an arbitrary set $S$ intersected with the set of points $z$ in $X$ such that $M(z) = b \mod \prod_{p \leq P} p$ then we conclude
\[ \mathbb{P}( (x_1,\ldots, x_m) \in S \ | \ (y_p)_{p \leq P} = b) = \mathbb{P}( (x_{i +1},\ldots, x_{i+m}) \in S \ | \ (y_p)_{p \leq P} = b + i). \] 
For each natural number $P$ and each prime $P/2 < p \leq P$, for at most two values of $i \leq P$ is it true that $y_p = i \mod p$. Therefore, at most two terms in the sum 
\[ \E_{i \leq P} \ p \ch_{y_p = i} \sup_{y \not\in B} \Big| \E_{h \leq k} x_{hp + i} f(T y) \Big| \]
are nonzero. Therefore $f_{P, p}$ is bounded by $2$. 
Let $W_P$ be a random variable with the same distribution as $Y_P$. Then by Theorem \ref{tao}
\[
\liminf_{P \rightarrow \infty} \E [ | f_P(X_P, Y_P) - f_P(X_P, W_P) | ] = 0.
\]
Since, for any natural number $P$,
\[
\E f_P(X_P, Y_P) > c + \eta - \e.
\]
We conclude that 
\[
\limsup_{P \rightarrow \infty} \E f_P(X_P, W_P) > c +  \eta - \e.
\]
Unpacking definitions, this proves
		\[
		\limsup_{P \rightarrow \infty} \E_{P/2 < p \leq P} \E_{i \leq P} \int_{X} \sup_{y \not\in B} \Big| \E_{h \leq k} \phi( T^{ph + i} x) \cdot f'(T^h y) \Big|  \mu(dx) > c + \eta - \e.
		\]
By the triangle inequality
		\[
		\limsup_{P \rightarrow \infty} \E_{P/2 < p \leq P} \E_{i \leq P} \int_{X}  \sup_{y \not\in B} \Big| \E_{h \leq k} \E^{\mu} [f | \z] ( T^{ph + i} x) \cdot f'(T^h y) \Big|  \mu(dx) > c + \eta - 2\e.
		\]
Since $\e$ was arbitrary
		\[
		\limsup_{P \rightarrow \infty} \E_{P/2 < p \leq P} \E_{i \leq P} \int_{X}  \sup_{y \not\in B} \Big| \E_{h \leq k} \E^{\mu} [f | \z] ( T^{ph + i} x) \cdot f'(T^h y) \Big|  \mu(dx) > c.
		\]
By translation invariance
		\[
		\limsup_{P \rightarrow \infty} \E_{P/2 < p \leq P} \int_{X}  \sup_{y \not\in B} \Big| \E_{h \leq k} \E^{\mu} [f | \z] ( T^{ph} x) \cdot f'(T^h y) \Big|  \mu(dx) > c.
		\]
		This completes the proof.
	\end{proof}
\subsection{Nilsystems and Algebraic Structure}
	Now we want to use \cite{HostKra} to show that $\E^{\mu} [f | \z]$ has some local algebraic structure. This algebraic structure makes $\E^{\mu} [f | \z]$ much easier to understand than $f$.
	\begin{prop}\label{fourierdecomp}
		Let $\omega$ be an element of $\Omega$ such that 
		\[
		\limsup_{P \rightarrow \infty} \E_{P/2 < p \leq P} \int_X \sup_{y \not\in B} \Big| \E_{h \leq k} \E^{\mu} [f | \z] (T^{ph} x) \cdot f'(T^h y) \Big|  \mu_\omega(dx) > c.
		\]
		Then for almost all such choices for $\omega$, there exists a collection of nilsystems $(G(j)/ \Gamma(j), dx, g(j), \mathcal{B} )$, $1$-bounded functions $F_j$ and factor maps $\psi_j \colon X \rightarrow G(j) / \Gamma(j)$ so that $F_j$ is a nilcharacter on $G(j) / \Gamma(j)$ with frequency nontrivial on the identity component and such that, after identifying $F_j$ with a function on $X$, we have that $\sum F_j = F$ satisfies $|| F ||_\infty \leq 1$ and 
		\[
		\limsup_{P \rightarrow \infty} \E_{P/2 < p \leq P} \int_X \sup_{y \not\in B} \Big| \E_{h \leq k} F(T^{ph} x) \cdot f'(T^h y) \Big|  \mu_\omega(dx) > c.
		\]
	\end{prop}
	\begin{proof}
		We are given that 
		\[
		\limsup_{P \rightarrow \infty} \E_{P/2 < p \leq P} \int_X \sup_{y \not\in B} \Big| \E_{h \leq k} \E^{\mu} [f | \z] (T^{ph} x) \cdot f'(T^h y) \Big|  \mu_\omega(dx) > c.
		\]
		Recall that by Lemma \ref{conditionwell}, we know that 
		\[
		\limsup_{P \rightarrow \infty} \E_{P/2 < p \leq P} \int_X \sup_{y \not\in B} \Big| \E_{h \leq k} \E^{\mu_\omega} [f | \z_\omega] (T^{ph} x) \cdot f'(T^h y) \Big|  \mu_\omega(dx) > c.
		\]
		By \cite{HostKra} Theorem 10.1, $(X, \mu_\omega, T, \z)$ is isomorphic to an inverse limit of nilsystems. Therefore, there exists $(G / \Gamma, dx, g, \mathcal{B})$ a nilsystem, $\psi \colon X \rightarrow G / \Gamma$ a factor map such that 
		\[
		\limsup_{P \rightarrow \infty} \E_{P/2 < p \leq P} \int_X \sup_{y \not\in B} \Big| \E_{h \leq k} \E^{\mu_\omega} [f | \psi^{-1}(\mathcal{B}) ] (T^{ph} x) \cdot f'(T^h y) \Big|  \mu_\omega(dx) > c.
		\]
		We denote $F = \E^{\mu_\omega} [f | \psi^{-1}(\mathcal{B}) ]$. By a Fourier decomposition, we may write $F$ as a sum of nilcharacters, $F = \sum_\xi F_\xi$. For each $\xi$, either $\xi$ is nontrivial on the identity component of $G / \Gamma$ or $\xi$ is trivial on the identity component. If $\xi$ is trivial on the identity component and the step $s$ of $G$ is $> 1$, then $\xi$ is actually trivial on $G_s$. That is because, for any $\sigma$ in $G$, the multiplication by $\sigma$ map $\sigma \colon G / \Gamma \rightarrow G / \Gamma$ is continuous so it takes components to components. Let $\sigma_* \colon \text{components of $G$} \rightarrow \text{components of $G$}$ be the induced map on components and let $\tau$ be any other element of $G$. Then if $\sigma$ and $\tau$ are in the same component of $G$ then for any $\sigma'$ in $G$, multiplication by $\sigma'$ on the right is also continuous, so $\sigma \sigma'$ is in the same component as $\tau \sigma'$ so $\sigma_* = \tau_*$. We return to the general case where $\sigma$ and $\tau$ are not necessarily in the same component. Also note that, for any element $\gamma$ in $\Gamma$, $(\gamma \sigma)_* = [\gamma, \sigma]_* \sigma_*  \gamma_* = [\gamma, \sigma]_* \sigma_*$. Pick $n$, $m$, $\gamma$ and $\gamma'$ such that $g^n \gamma$ is in the same component as $\sigma$ and $g^m \gamma'$ is in the same component as $\tau$. Thus $[\sigma, \tau]_* = [g^n \gamma, g^m \gamma']_* = \pi_* [g^n, g^m]_*$ where $\pi$ is an element of higher order. Of course $[g^n, g^m] = e$ and by induction we get that $[\sigma, \tau]_*$ is the identity and therefore $[\sigma, \tau] \gamma$ is in the identity component for some $\gamma$. Therefore, if $s > 1$, the function $F_\xi$ descends to a function of on $(G / G_s) / (\Gamma / \Gamma_s)$. By induction, we can almost prove the theorem, namely we can find a collect of nilsystems $(G(j)/ \Gamma(j), dx, g(j), \mathcal{B} )$ and functions $F_j$ and factor maps $\psi_j \colon X \rightarrow G(j) / \Gamma(j)$ so that $F_j$ is a nilcharacter on $G(j) / \Gamma(j)$ with frequency nontrivial on the identity component or $G(j)$ is abelian and such that, after identifying $F_j$ with a function on $X$, we have that $\sum F_j = F$ satisfies $|| F ||_\infty \leq 1$ and 
		\[
		\limsup_{P \rightarrow \infty} \E_{P/2 < p \leq P} \int_X \sup_{y \not\in B} \Big| \E_{h \leq k} F(T^{ph} x) \cdot f'(T^h y) \Big|  \mu_\omega(dx) > c.
		\]
		It remains to observe that the case of a locally constant function on an abelian group cannot occur by Corollary \ref{MRTerg} as follows: we can think of the $F_j's$ as all functions on the group $G$ with some additional equivariance properties; by construction the different $F_j$'s have different frequencies so if $F_r$ is a locally constant function on an abelian group and thus is locally periodic, meaning $F_r(T^h x)$ is a periodic function of $h$. Then by Corollary \ref{MRTerg},
		\begin{align*}
		0 = &\int_{X} f \cdot \overline{F_r} \mu_\omega(dx).
		\intertext{Since $F_r$ is $\psi^{-1}(\mathcal{B})$ measurable,}
		= &\int_{G / \Gamma} F \cdot \overline{F_r} dx.
		\intertext{Since all the $F_j$'s have different frequencies, they are all orthogonal to each other.}
		= &\int_{G / \Gamma} F_r \cdot \overline{F_r} dx.
		\end{align*}
	\end{proof}
\begin{rem}
		Note that if we also know the $\kappa-1$-Fourier uniformity conjecture then the step of all nilpotent Lie groups is $\geq \kappa$ by Proposition \ref{conderg} (plugging in $\overline{F_r} = \phi$ in the statement of that proposition).
\end{rem}

	\begin{cor}\label{fixG}
		There exists a natural number $L$ independent of $k$, a nilpotent Lie group $G$ of step $s$, a cocompact subgroup $\Gamma$, an ergodic element $g$ in $G$ and a nilcharacter $\F$ with nontrivial frequency even when restricted to the identity component,
		\[
		\limsup_{k \in \mathcal{K}} \limsup_{P \rightarrow \infty} \E_{P/2 < p \leq P} \int_X \sup_{y \not\in B} \Big| \E_{h \leq k} \F(g^{ph} x) \cdot f'(T^h y) \Big|  dx > \frac{c}{L},
		\]
where $\K$ as defined in Subsection \ref{background} is an infinite set such that for $k$ in $\K$, the number of words of length $k$ of $f'$ is $o(k^2)$ if $t = 2$ or $O(k^{t-\e})$ if $t \neq 2$ for some $\e$.
	\end{cor}
	\begin{proof}
By Corollary \ref{cor1},
		\[
		\limsup_{P \rightarrow \infty} \E_{P/2 < p \leq P} \int_X \sup_{y \not\in B} \Big| \E_{h \leq k} \E^{\mu} [f | \z] (T^{ph} x) \cdot f'(T^h y) \Big|  \mu(dx) > c.
		\]
Thus, for a positive measure set of $\omega$,
		\[
		\limsup_{P \rightarrow \infty} \E_{P/2 < p \leq P} \int_X \sup_{y \not\in B} \Big| \E_{h \leq k} \E^{\mu} [f | \z] (T^{ph} x) \cdot f'(T^h y) \Big|  \mu_\omega(dx) > c.
		\]
		By Proposition \ref{fourierdecomp}, we know that for almost every $\omega$,
		\[
		 \limsup_{k \in \K}  \limsup_{P \rightarrow \infty} \E_{P/2 < p \leq P} \int_X \sup_{y \not\in B} \Big| \E_{h \leq k} \sum_j F_j (T^{ph} x) \cdot f'(T^h y) \Big|  \mu_\omega(dx) > c,
		\]
where $F_j$ is as in Proposition \ref{fourierdecomp}.
Fix such an $\omega$.

		Since the sum $\sum_j F_j$ converges in $L^2(\mu_\omega)$, there exists a natural number $L$ independent of $k$ such that $|| \sum_{j \leq \frac{L}{2}} F_j - \sum_j F_j ||_{L^2(\mu_\omega)} < \frac{c}{2}$. By the triangle inequality
		\begin{align*}
		 \limsup_{k \in \K}  \limsup_{P \rightarrow \infty} \E_{P/2 < p \leq P} & \int_X \sup_{y \not\in B} \Big| \E_{h \leq k} \sum_{j \leq \frac{L}{2}} F_j (T^{ph} x) \cdot f'(T^h y) \Big|  \mu_\omega(dx) \\ 
+ & \int_X \sup_{y \not\in B} \E_{h \leq k} \Big|  \sum_{j > \frac{L}{2}} F_j (T^{ph} x) \cdot f'(T^h y) \Big|  \mu_\omega(dx) 
> c.
		\end{align*}
The second term is bounded by
\[
\int_X \Big|  \sum_{j > \frac{L}{2}} F_j (x) \Big|  \mu_\omega(dx),
\]
using that $\mu_\omega$ is shift invariant and $f'$ is 1-bounded.
By Cauchy-Schwarz, this term is bounded by $|| \sum_{j \leq \frac{L}{2}} F_j - \sum_j F_j ||_{L^2(\mu_\omega)} < \frac{c}{2}$. Thus by the triangle inequality.
		\[
		\limsup_{k \in \K}  \limsup_{P \rightarrow \infty} \E_{P/2 < p \leq P} \int_X \sup_{y \not\in B} \Big| \E_{h \leq k} \sum_{j \leq \frac{L}{2}} F_j (T^{ph} x) \cdot f'(T^h y) \Big|  \mu_\omega(dx) > \frac{c}{2}.
		\]
		By the pigeonhole principle, there exists some $F_j$ such that 
		\[
		\limsup_{k \in \K}  \limsup_{P \rightarrow \infty} \E_{P/2 < p \leq P} \int_X \sup_{y \not\in B} \Big| \E_{h \leq k} F_j (T^{ph} x) \cdot f'(T^h y) \Big|  \mu_\omega(dx) > \frac{c}{L}.
		\]
		Renaming everything gives the conclusion. We remark that the corollary just stated that such an $L$, $G$, $\Gamma$, $g$ and $\Phi$ exist and therefore the statement of the corollary allows $L$ to depend on $G$, $\Gamma$ and all the other data that comes from $\omega$. The remainder of the argument essentially takes place inside a single ergodic component and so how the constants vary from component to component is not important for our purposes.
	\end{proof}
	For the remainder of the proof, we fix $G$, $\Gamma$, $g$ and $\F$.
	We let $c_0 = \frac{c}{L}$.
	For the next few pages, we fix an integer $k$ in $\K$ such that 
	\[
	\limsup_{P \rightarrow \infty} \E_{P/2 < p \leq P} \int_X \sup_{y \not\in B} \Big| \E_{h \leq k} \F(g^{ph} x) \cdot f'(T^h y) \Big|  dx > c_0.
	\]
	We will later send $k$ to infinity.
	The following lemma does two things: First, it uses H\"older's inequality to raise the exponent of $\Big| \E_{h \leq k} \F(g^{ph} x) \cdot f'(T^h y) \Big|$. We want this term raised to an even power because we want to expand out the product and get rid of the absolute values which are less ``algebraic" and therefore harder to understand directly using the theory of nilpotent Lie groups. We also want this even power to be larger the more oscillatory our function $\F$ is. This is because the more $\F$ oscillates, the more cancellation we expect in larger and larger products. The larger the power we use, the smaller the fraction of terms which do not exhibit cancellation is. Second, we use the pigeonhole principle. This lemma and the following lemma are where we make essential use of our bound on the word growth rate of $b$.
	\begin{lem}\label{lem1}
		Recall that $b$ had at most $k^{t - \e}$ words of length $k$ occuring with  positive upper logarithmic density if for $k$ in $\K$ or $b$ has $o(k^2)$ many words of length $k$ that occur with positive upper logarithmic density if $t = 2$. Fix $\delta$ a constant that is small even when compared to $c_0$.
		Then for each $k$ in $\K$ there is a word $\epsilon = (\epsilon_1, \ldots, \epsilon_k)$ of length $k$ such that
		\[
		\limsup_{P \rightarrow \infty} \E_{P/2 < p \leq P} \int_X  \Big| \E_{h \leq k} \F(g^{ph} x) \cdot \epsilon_h \Big|^{2t}  dx > k^{-t + \e} c_0^{2t},
		\]
		for $t \neq 2$ and when $t = 2$, we have
		\[
		\limsup_{P \rightarrow \infty} \E_{P/2 < p \leq P} \int_X  \Big| \E_{h \leq k} \F(g^{ph} x) \cdot \epsilon_h \Big|^{2t}  dx > \delta^{-1} k^{-t } c_0^{2t}.
		\]
	\end{lem}
	\begin{proof}
		We know that
		\[
		\limsup_{P \rightarrow \infty} \E_{P/2 < p \leq P} \int_X \sup_{y \not\in B} \Big| \E_{h \leq k} \F(g^{ph} x) \cdot f'(T^h y) \Big|  dx > c_0.
		\]
		By Holder's inequality, we have
		\[
		\limsup_{P \rightarrow \infty} \E_{P/2 < p \leq P} \int_X \sup_{y \not\in B} \Big| \E_{h \leq k} \F(g^{ph} x) \cdot f'(T^h y) \Big|^{2t}  dx > c_0^{2t}.
		\]
		Because each term is nonnegative, we can replace the essential $\sup$ by a sum over words that occur with positive $\log$-density.
		\[
		\sum_{\epsilon} \limsup_{P \rightarrow \infty} \E_{P/2 < p \leq P} \int_X  \Big| \E_{h \leq k} \F(g^{ph} x) \cdot \epsilon_h \Big|^{2t}  dx > c_0^{2t}.
		\]
		We assumed that the number of words occuring with positive logarithmic density and therefore the number of terms in the sum is at most $\delta k^t$ when $t = 2$ or $k^{t - \e}$ when $t \neq 2$.
		By the pigeonhole principle, when $t = 2$ there is a word such that
		\[
		\limsup_{P \rightarrow \infty} \E_{P/2 < p \leq P} \int_X  \Big| \E_{h \leq k} \F(g^{ph} x) \cdot \epsilon_h \Big|^{2t}  dx > \delta^{-1} k^{-t} c_0^{2t},
		\]
		and similarly for $t \neq 2$, we have 
		\[
		\limsup_{P \rightarrow \infty} \E_{P/2 < p \leq P} \int_X  \Big| \E_{h \leq k} \F(g^{ph} x) \cdot \epsilon_h \Big|^{2t}  dx > k^{-t + \e} c_0^{2t},
		\]
		which completes the proof.
	\end{proof}
	We need a slightly different estimate for the abelian case. The key to the next lemma is the idea that if $e(\alpha h)$ correlates with $\epsilon_h$ for $h \leq k$ then $e(\alpha h)$ also must correlate with translates of $\epsilon$ of size $\sim k$. Thus, in the abelian case, the previous lemma is rather lossy. When we replace the $\sup$ by a sum, we should gain an extra power of $k$.
	\begin{lem}\label{lem2}
		For $t =2$, for all $k$ in $\K$, there is a word $\epsilon$ such that
		\begin{align*}
		&\limsup_{P \rightarrow \infty} \int_X \E_{P/2 < p \leq P}  \sup_{\ell \in \N} \Big| \E_{h \leq k}  \F(g^{ph + \ell} x) \cdot \epsilon_h \Big|^{2t-2}  dx \\ > &\frac{c_0^{2t-2}}{9} \cdot \floor*{\frac{c_0^{2t-2} k}{6}} \delta^{-1} k^{-t}.
		\end{align*}
	\end{lem}
	\begin{proof}
		Again, we know that
		\[
		\limsup_{P \rightarrow \infty} \E_{P/2 < p \leq P} \int_X \sup_{y \not\in B} \Big| \E_{h \leq k} \F(g^{ph} x) \cdot f'(T^h y) \Big|  dx > c_0.
		\]
		Again, by Holder's inequality
		\[
		\limsup_{P \rightarrow \infty} \E_{P/2 < p \leq P} \int_X \sup_{y \not\in B} \Big| \E_{h \leq k} \F(g^{ph} x) \cdot f'(T^h y) \Big|^{2t-2}  dx > c_0^{2t-2}.
		\]
		Again we want to replace $F'(T^h y)$ by a sum over words. Let $P$ be a number satisfying  
		\[
		\E_{P/2 < p \leq P} \int_X \sup_{y \not\in B} \Big| \E_{h \leq k} \F(g^{ph} x) \cdot f'(T^h y) \Big|^{2t-2}  dx > c_0^{2t-2}.
		\]
		Let $A$ be the set of $x$ such that 
		\[
		\E_{P/2 < p \leq P} \sup_{y \not\in B} \Big| \E_{h \leq k} \F(g^{ph} x) \cdot f'(T^h y) \Big|^{2t-2} > \frac{2c_0^{2t-2}}{3}.
		\]
		Therefore, the measure of $A$ is at least $\frac{c^{2t-2}}{3}$.
		We want to show that for $\mu$-almost every $x$ in $A$, there are at least $\floor*{\frac{c^{2t-2} k}{6}}$ many distinct words of $f'$ such that 
		\[
		 \E_{P/2 < p \leq P} \sup_{\ell \in \N} \sup_{y \not\in B} \Big| \E_{h \leq k} \F(g^{ph + \ell} x) \cdot f'(T^h y) \Big|^{2t-2} > \frac{c_0^{2t-2}}{3}.
		\]
		Let $y$ be an element of $B^c$ such that the words of $f'(T^h y)$ are words of $Y$ and such that
		\[
		\E_{P/2 < p \leq P} \Big| \E_{h \leq k} \F(g^{ph} x) \cdot f'(T^h y) \Big|^{2t-2} > \frac{2c_0^{2t-2}}{3}.
		\]
		Denote by $\epsilon(m)$ the word of length $k$ whose $h^{th}$ entry is $\epsilon(m)_h = f'(T^{m + h} y)$. If the words $\epsilon(m)$ are distinct for $m = 1, \ldots, \floor*{\frac{c_0^{2s} k}{6}}$ then by the triangle inequality 
		\begin{align*}
		& \E_{P/2 < p \leq P} \sup_{y \not\in B} \Big| \E_{h \leq k} \F(g^{ph +pm} x) \cdot \epsilon(m)_h \Big|^{2t-2} \\ = & \E_{P/2 < p \leq P} \Big| \E_{h \leq k} \F(g^{ph +pm} x) \cdot f'(T^{h + pm} y) \Big|^{2t-2}. 
\intertext{We note that all but $2mk$ many terms in the average are the same if we replace $\F (g^{ph +pm} x) \cdot f'(T^{h + pm} y)$ with $\F (g^{ph} x) \cdot f'(T^{h} y)$. Thus }
		\geq& \E_{P/2 < p \leq P} \Big| \E_{h \leq k} \F(g^{ph} x) \cdot f'(T^{h } y) \Big|^{2t-2} - \frac{2m}{k} \\
		\geq& \E_{P/2 < p \leq P} \Big| \E_{h \leq k} \F(g^{ph} x) \cdot f'(T^{h } y) \Big|^{2t-2} - \frac{c_0^{2t-2}}{3}.
		\end{align*}
		Suppose for a moment that instead the words $\epsilon(m)$ are not distinct for $m = 1, \ldots, \floor*{\frac{c_0^{2t-2} k}{6}}$. Then there exist a minimum $j$ such that $\epsilon(1), \ldots, \epsilon(j)$ are not distinct. Fix such a $j$ for the remainder of the proof. Thus, there exists some $1 \leq d < j$ such that $\epsilon(j) = \epsilon(j-d)$. We claim that $\epsilon(j-d)$ is $d$-periodic: that's because $\epsilon(j-d)_h = \epsilon(j)_h = \epsilon(j-d)_{h+d}$. Furthermore, if $\epsilon(j-d-1)_1 = \epsilon(j-1)_1$ then since $\epsilon(j-d-1)_h = \epsilon(j-d)_{h-1} = \epsilon(j)_{h-1} = \epsilon(j-1)_h$ for all $h > 1$, we clearly have $\epsilon(j-d-1) = \epsilon(j-1)$ and $j$ is not minimal. For the rest of the proof, let $r$ be the minimum number such that $r \geq j-d$ and $\epsilon(r)$ is not $d$ periodic. For $y$ not in $B$, we can find such an $r$ because $f'(T^h y)$ is not eventually periodic. Since $\epsilon(r)$ is not $d$ periodic but $\epsilon(r-1)$ is $d$ periodic and is equal to $\epsilon(q)$ for some $q$ between $j-d$ and $j-1$, we have that $\epsilon(r)_k \neq \epsilon(r)_{k-d}$ but $\epsilon(r)_h = \epsilon(r)_{h-d}$ for all other $h \leq k$. We claim that the words $1, \ldots, j-1$ and $r, \ldots, r + \floor*{\frac{c_0^{2t-2} k}{6}} - j + 1$ are all distinct. The reason is that for all $m$ between $1$ and $j-d$, we have that $\epsilon(m)_h = \epsilon(j-d)_{h + m - j + d}$ for all $h > m - j + d$ and precisely no larger range of $h$ and for all $m$ between $r$ and $r + \floor*{\frac{c_0^{2t-2} k}{6}} - j + 1$ we have that $\epsilon(m)_h = \epsilon(j-d)_{h + m - r}$ for all $1 \leq h < k - m + r$ and precisely no larger range of $h$. For $m$ between $j-d$ and $j-1$, $\epsilon(m)$ is $d$ periodic but because $j$ was the minimal natural number such that $\epsilon(1), \ldots, \epsilon(j)$ are not distinct, we have that the $\epsilon(m)$ for $m$ between $j-d$ and $j-1$ are still distinct. For $m$ between $1$ and $j-d$ and $m'$ between $r$ and $r + \floor*{\frac{c_0^{2t-2} k}{6}} - j + d$ we have that the intervals $h > m- j + d$ and $h < k -m + r$ meet so the previous argument shows that $\epsilon(m) \neq \epsilon(m')$. A similar triangle inequality computation shows that for $m$ between $r$ and $r + \floor*{\frac{c_0^{2t-2} k}{6}} - j + 1$ we still have
		\begin{align*}
		& \E_{P/2 < p \leq P} \sup_{\ell \in \N}  \Big| \E_{h \leq k} \F(g^{ph +\ell} x) \cdot \epsilon(m)_h \Big|^{2t-2} \\ =& \E_{P/2 < p \leq P} \sup_{\ell \in \N} \Big| \E_{h \leq k} \F(g^{ph + \ell} x) \cdot f'(T^{h + m} y) \Big|^{2t-2} \\
		\geq&  \E_{P/2 < p \leq P} \sup_{\ell \in \N} \Big| \E_{h \leq k} \F(g^{ph + \ell} x) \cdot f'(T^{h + r -1} y) \Big|^{2t-2} - \frac{2(m - r + 1)}{k}
		\intertext{Next, we use that $\epsilon(r-1) = \epsilon(q)$ for some $q$ between $j-d$ and $j-1$.}
		\geq& \E_{P/2 < p \leq P} \sup_{\ell \in \N} \Big| \E_{h \leq k} \F(g^{ph + \ell} x) \cdot f'(T^{h + q} y) \Big|^{2t-2} - \frac{2(m - r + 1)}{k} \\
		\geq& \E_{P/2 < p \leq P} \Big| \E_{h \leq k} \F(g^{ph} x) \cdot f'(T^{h} y) \Big|^{2t-2} - \frac{c_0^{2t-2}}{3}.
		\end{align*}
		This proves the claim that for $x$ in $A$ there are at least $\floor*{\frac{c_0^{2t-2} k}{6}}$ many distinct words $\epsilon$ of $f'$ such that 
		\[
		\E_{P/2 < p \leq P} \sup_{\ell \in \N}  \Big| \E_{h \leq k} \F(g^{ph +\ell} x) \cdot \epsilon_h \Big|^{2t-2} >  \frac{c_0^{2t-2}}{3}.
		\]
		Summing over words we get that for almost every $x$ in $A$,
		\begin{align*}
		&\sum_{ \epsilon \text{ a word of $f'$} } \E_{P/2 < p \leq P} \sup_{\ell \in \N} \Big| \E_{h \leq k} \F(g^{ph + \ell} x) \cdot \epsilon_h \Big|^{2t-2} \\ >& \floor*{\frac{c^{2t-2} k}{6}} \cdot \frac{c_0^{2t-2}}{3}.
		\end{align*}
		Next, we use that $\mu(A) > \frac{c_0^{2t-2}}{3}$.
		\begin{align*}
		&\int_X \sum_{ \epsilon \text{ a word of $f'$} } \sup_{\ell \in \N} \E_{P/2 < p \leq P} \Big| \E_{h \leq k} \F(g^{ph + \ell} x) \cdot \epsilon_h \Big|^{2t-2} \\ >& \frac{c_0^{2t-2}}{3} \floor*{\frac{c_0^{2t-2} k}{6}} \cdot \frac{c_0^{2t-2}}{3}.
		\end{align*}
		Sending $P$ to infinity and using the pigeonhole principle, we deduce that for some word $\epsilon$,
		\begin{align*}
		&\limsup_{P \rightarrow \infty} \int_X \sup_{\ell \in \N} \E_{P/2 < p \leq P} \Big| \E_{h \leq k} \F(g^{ph + \ell} x) \cdot \epsilon_h \Big|^{2t-2}  dx \\ > &\frac{c_0^{2t-2}}{9} \cdot \floor*{\frac{c_0^{2t-2} k}{6}} \delta^{-1} k^{-t}.
		\end{align*}
	\end{proof}
	\begin{rem}\label{rem1}
		For $G$ abelian and therefore $\F$ a character, we have
		\begin{align*}
		&\limsup_{P \rightarrow \infty} \int_X \E_{P/2 < p \leq P} \sup_{\ell \in \N} \Big| \E_{h \leq k} \F(g^{ph + \ell} x) \cdot \epsilon_h \Big|^{2t-2}  dx \\ = & \limsup_{P \rightarrow \infty} \int_X  \E_{P/2 < p \leq P} \Big| \E_{h \leq k}  \F(g^{ph} x) \cdot \epsilon_h \Big|^{2t-2} dx
		\end{align*}
	\end{rem}
	Therefore, by choosing $\delta$ sufficiently small, in the abelian case, we get
	\[
	\limsup_{k \in \K} \limsup_{P \rightarrow \infty} \E_{P/2 < p \leq P} \int_X | \E_{h \leq k} \epsilon_h \Phi(g^{ph} x) |^{2} dx \gg k^{-1}.
	\]
	
	The next theorem contradicts the previous two lemmas and proves Theorem \ref{mainthm}. In its proof, we rely heavily on \cite{MR3548534}, \cite{Frantzikinakis}, \cite{GT2}, \cite{GT1} and \cite{GTZ}.
	\begin{thm}\label{bigthm}
		Recall that, after Corollary \ref{fixG}, we fixed a nilpotent Lie group $G$, a cocompact lattice $\Gamma$, a nilcharacter $\F$ with with nontrivial frequency on the identity component $\xi$ and an element $g$ which acts ergodically on $ G / \Gamma$, such that $||\F||_{L^\infty_x} = 1$.
		Recall that the step $s$ of $G$ is at least $\kappa$ where $t = {\kappa + 1 \choose 2}$.
		Let $\epsilon$ be a sequence of words implicitly depending on $k$. Let $\e > 0$. Then
		\[
		\limsup_{k \in \K} \limsup_{P \rightarrow \infty} \int_X \E_{P/2 < p \leq P} | \E_{h \leq k} \epsilon_h \Phi(g^{ph} x) |^{2t} dx \cdot k^{t - \e} = 0,
		\]
		If $t = 2$ then we do not need the epsilon loss and instead get the estimate
		\[
		\limsup_{k \in \K} \limsup_{P \rightarrow \infty}\int_X  \E_{P/2 < p \leq P} | \E_{h \leq k} \epsilon_h \Phi(g^{ph} x) |^{2t} dx \cdot k^{t} \leq C_s.
		\]
	\end{thm}
This contradicts Lemmas \ref{lem1} and \ref{lem2} as follows. When $G$ is abelian and thus $t = 2$, Lemma \ref{lem2} states that there is a word $\epsilon$ of length $k$ such that
		\begin{align*}
		&\limsup_{P \rightarrow \infty} \int_X \E_{P/2 < p \leq P}  \sup_{\ell \in \N} \Big| \E_{h \leq k}  \F(g^{ph + \ell} x) \cdot \epsilon_h \Big|^{2t-2}  dx \\ > &\frac{c_0^{2t-2}}{9} \cdot \floor*{\frac{c_0^{2t-2} k}{6}} \delta^{-1} k^{-t},
		\end{align*}
for any $\delta$ we choose so long as $k$ is chosen from the set $\K$ of natural numbers such that $f'$ has fewer than $\delta k^2$ many words of length $k$. Thus, picking $\delta$ small, (in particular, smaller than say $C_1^{-1} 100 c_0^{4}$), we find that
		\begin{align*}
		\limsup_{P \rightarrow \infty} \int_X \E_{P/2 < p \leq P}  \Big| \E_{h \leq k}  \F(g^{ph} x) \cdot \epsilon_h \Big|^{2t-2}  dx  > k^{-t} C_1,
		\end{align*}
contradicting Theorem \ref{bigthm}. (Note that we have replaced $|\E_{h \leq k} \F(g^{ph+\ell}) \epsilon_h|$ by the same expression without the shift in $\ell$ as in Remark \ref{rem1}). Similarly, if the group is not abelian, Lemma \ref{lem1} states that there exists a word $\epsilon$ of length $k$ such that
\[
		\limsup_{P \rightarrow \infty} \E_{P/2 < p \leq P} \int_X  \Big| \E_{h \leq k} \F(g^{ph} x) \cdot \epsilon_h \Big|^{2t}  dx > k^{-t + \e} c_0^{2t},
		\]
		for $t \neq 2$ and when $t = 2$, we have
		\[
		\limsup_{P \rightarrow \infty} \E_{P/2 < p \leq P} \int_X  \Big| \E_{h \leq k} \F(g^{ph} x) \cdot \epsilon_h \Big|^{2t}  dx > \delta^{-1} k^{-t } c_0^{2t}.
		\]
When $t = 2$, again by picking $\delta$ small, this time smaller than $C_s^{-1} c_0^{2t}$, proves that
		\[
		\limsup_{P \rightarrow \infty} \E_{P/2 < p \leq P} \int_X  \Big| \E_{h \leq k} \F(g^{ph} x) \cdot \epsilon_h \Big|^{2t}  dx > C_s k^{-t },
		\]
again contadicting Theorem \ref{bigthm}. 
Finally, when $t \neq 2$, 
\[
		\limsup_{P \rightarrow \infty} \E_{P/2 < p \leq P} \int_X  \Big| \E_{h \leq k} \F(g^{ph} x) \cdot \epsilon_h \Big|^{2t}  dx > k^{-t + \e} c_0^{2t},
		\]
contradicts
		\[
		\limsup_{k \in \K} \limsup_{P \rightarrow \infty} \int_X \E_{P/2 < p \leq P} | \E_{h \leq k} \epsilon_h \Phi(g^{ph} x) |^{2t} dx \cdot k^{t - \e} = 0,
		\]
from Theorem \ref{bigthm}.

Thus, the rest of this section will be devoted to showing that Theorem \ref{bigthm} is true.
	Suppose not and for the moment fix $k$ in $\K$  such that
	\[
	\limsup_{P \rightarrow \infty} \E_{P/2 < p \leq P} \int_X \Big| \E_{h \leq k} \epsilon_h \Phi(g^{ph} x) \Big|^{2t} dx \gg k^{-t+\e}.
	\]
	The first step is to replace averages over primes by uniform averages over natural numbers. To do this, we need the machinery of Green-Tao \cite{GT1} \cite{GT2} and Green-Tao-Ziegler \cite{GTZ}.
	By the triangle inequality, we may replace averages over primes by averages weighted by the von Mangoldt function.
	\[
	\limsup_{P \rightarrow \infty} \E_{P/2 < n \leq P} \int_X \Lambda(n) \Big| \E_{h \leq k} \epsilon_h \Phi(g^{ph} x) \Big|^{2t} dx \gg k^{-t + \e}.
	\]
	We denote $\psix(m) = \Phi(g^{m} x)$.
	We expand:
	\begin{align*}
	\limsup_{P \rightarrow \infty} & \E_{P/2 < n \leq P} \int_X \Lambda(n) \E_{J \in [k]^{2t}} \epsilon_J  \psix(nj_1) \cdot \cdots \cdot \psix(nj_t) \cdot  \\ &\overline{\psix}(nj_{t+1}) \cdot \cdots \cdot \overline{\psix}(nj_{2t} ) dx
	\gg k^{-t + \e},	
	\end{align*}
where $\epsilon_j$ is a phase given by the formula $\epsilon_J = \epsilon_{j_1} \cdots \epsilon_{j_t} \cdot \overline{\epsilon}_{j_{t+1}} \cdots \overline{\epsilon}_{j_{2t}}$.

	We say $J \in [k]^s$ is diagonal if $\# \{ m \leq s \colon j_m = h \} = \# \{ m > s \colon j_m = h \}$ for all $h \leq k$. We say $J$ solves the Vinogradov mean value problem if, for all $m$ between $1$ and $s$, we have
	\[
	j_1^m + \cdots + j_t^m = j_{t+1}^m + \cdots + j_{2t}^m.
	\] 
	Every diagonal $J$ also solves Vinogradov's mean value problem. 
	We rely on the following Theorem due to Bourgain, Demeter and Guth which says that those account for ``most" solutions, up to a constant.
	\begin{thm}[\cite{MR3548534} Theorem 1.1]\label{BDG}
		For all $\e$ and $s$ there exists a constant $C_{s,\e}$ such that the number of solutions to the Vinogradov mean value problem is less than $C_{s,\e}k^{t + \e}$ where $t \leq {{s+1} \choose 2}$.
	\end{thm}
	We will show that if $J$ does not solve Vinogradov's mean value problem then $J$ does not contribute to the sum. Thus, fix $J$ which does not solve Vinogradov's mean value theorem and suppose that
	\begin{align*}
	&\limsup_{P \rightarrow \infty}  \left| \E_{P/2 < n \leq P} \int_X \Lambda(n) \psix(nj_1) \cdot \cdots \cdot \psix(nj_t) \cdot \overline{\psix}(nj_{t+1}) \cdot \cdots \cdot \overline{\psix}(nj_{2t} ) dx \right| \\ &\gtrsim_{k, s, c}  1.	
	\end{align*}
	We denote 
	\[\Psix (n) = \psix(nj_1) \cdot \cdots \cdot \psix(nj_t) \cdot \overline{\psix}(nj_{t+1}) \cdot \cdots \cdot \overline{\psix}(nj_{2t}). \]
	Fix a subsequence such that
	\[
	\lim_{P \in I} \left|  \E_{P/2 < n \leq P} \int_X \Lambda(n) \Psix(n) dx \Big| =  \limsup_{P \rightarrow \infty} \Big|  \E_{P/2 < n \leq P} \int_X \Lambda(n) \Psix(n)dx \right| \gtrsim_{k, s, c} 1,
	\]
	where $I$ is some infinite subset of the natural numbers and where the implied constant may depend on $\Psix$.
	Fix a large number $W$, a product of many small primes. We will later choose exactly how large $W$ must be. We pass to a subsequence where the following limit exists for each $b \leq W$,
	\[
	\lim_{P \in I'} \left| \E_{P/2 < Wn \leq P} \int_X \Lambda(Wn + b) \Psix(Wn +b) dx\right|,
	\]
	where $I'$ is an infinite subset of $I$.
	We may do this by a diagonalization argument.
	By the triangle inequality,
	\[
	\E_{b < W} \lim_{P \in I'} \left| \E_{P/2 < Wn + b \leq P} \int_X \Lambda(Wn + b) \Psix(Wn +b) dx \right| \gtrsim_{k, s, c} 1,
	\]
	where the implied constant does not depend on $W$. Note that because $b < W$, we miss at most one term by changing the bounds of the sum from $P/2 < Wn + b \leq P$ to $P/2 < Wn \leq P$. Since $W$ is much smaller than $P$, this is an acceptable error.
	Note that if $b$ is not coprime to $W$, then 
	\[ \lim_{P \in I'} \E_{P/2 < Wn \leq P}  \int_X \Lambda(Wn + b) \Psix(Wn +b) dx = 0, \]
	because $Wn + b$ is never prime.
	By the pigeonhole principle, there exists $b < W$ such that 
	\[ \lim_{P \in I'} \left| \E_{P/2 < Wn \leq P} \int_X \frac{W}{\varphi(W)}\Lambda(Wn + b) \Psix(Wn +b) dx \right| \gtrsim_{k, s, c} 1, \]
	where again the implied constant does not depend on $W$ and where $\varphi(W)$ is Euler's torient function, the function which counts the number of residue classes mod $W$ that are coprime to $W$. Denote $\frac{W}{\varphi(W)}\Lambda(Wn + b) = \Lambda_{b, W}$.
	Then we can write our expression as a sum of two terms
	\begin{align*}
	&\lim_{P \in I'} \left| \E_{P/2 < Wn \leq P} \int_X \Lambda_{b, W}(n)\Psix(Wn + b) dx \right| \\ = &\lim_{P \in I'}   \left| \E_{P/2 < Wn \leq P} \int_X (\Lambda_{b, W}(n) - 1)\Psix(Wn + b) + \Psix(Wn + b) dx \right|
	\end{align*}
	To handle the first term, we need the following theorems of Green-Tao and Green-Tao-Ziegler.
	\begin{thm}[\cite{GT2} Proposition 11.2]\label{gt1}
		Let $G / \Gamma$ be a degree $s$ filtered nilmanifold, and let $M > 0$. Suppose that $F(g^nx)_{n=1}^\infty$ is a bounded nilsequence on $G/\Gamma$ with Lipschitz constant at most $M$, where $F$ is a function on $G / \Gamma$, $g$ is an element of $G$ and $x$ is a point in $G/ \Gamma$. Let $\e \in (0,1)$ and $P$ a large natural number. Then we may decompose
		\[
		F(g^nx) = F_1(n) + F_2(n),
		\]
		where $F_1 \colon \N \rightarrow [-1,1]$ is a sequence with Lipschitz constant $O_{M,\e, G/ \Gamma}(1)$ and obeying the dual norm bound
		\[
		|| F_1 ||_{U^{s+1} [P/2 < Wn \leq P]^*} = O_{M,\e,G/\Gamma}(1),
		\]
		while $F_2 \colon \N \rightarrow \mathbb{R}$ obeys the uniform bound
		\[
		||F_2||_\infty \leq \e.
		\] 
	\end{thm}
	Note that the bound $|| F_1 ||_{U^{s+1} [N]^*} = O_{M,\e,G/\Gamma}(1)$ is uniform in the element $g$. We also need the following theorem of Green-Tao-Ziegler. The proof of this theorem is spread out over \cite{GTZ}, \cite{GT2} and \cite{GT1}, making it somewhat hard to give a specific theorem number. Essentially, if the Gowers norm were big then the Inverse Conjecture for the Gowers Norms would imply that the Mobius function correlates with a nilsequence which it does not by the Mobius-Nilsequence Conjecture. In \cite{GT2}, Theorem 7.2 states the theorem follows from the Mobius-Nilsequence Conjecture and the Inverse Conjecture for the Gowers Norms. The first of these conjectures is an immediate consequence of Theorem 1.1 in \cite{GT1}. The second of these conjectures is Theorem 1.3 in \cite{GTZ}. 
	\begin{thm}[\cite{GTZ}; see also \cite{GT1} and \cite{GT2}]\label{gt2}
		With all the notation as before,
		\[|| \Lambda_{b, W} - 1 ||_{U^{s+1} [P/2 < Wn \leq P]} = o_{W \rightarrow \infty}(1). \]
	\end{thm}
Thus, our nilsequence $\Psi_x$ can be written as a sum $\Psi_x = F_1 + F_2$ where $F_1$ and $F_2$ implicitly depend on $x$ and enjoy the following properties. $F_2$ is uniformly small so 
	\[
	\limsup_{P \in I'} \left| \E_{P/2 < Wn \leq P} (\Lambda_{b, W}(n) - 1) F_1(Wn + b) \right| 
	\]
can be estimated by simply moving the absolute values inside. The remaining term is bounded in dual norm so 
\begin{align*}
(\Lambda_{b, W} - 1)(n) \cdot F_1(n) \leq || \Lambda_{b, W} - 1 ||_{U^{s+1} [P/2 < Wn \leq P] }\cdot || F_1 ||_{U^{s+1} [P/2 < Wn \leq P]^*}
\end{align*}
which tends to $0$. For a similar argument, see the proof of Proposition 10.2 in \cite{GT2}. It may also be possible to circumvent the use of \cite{GTZ} by using Theorem 7.1 in \cite{GT1}.
	Putting this together, we get that 
	\[
	\limsup_{P \in I'} \left| \E_{P/2 < Wn \leq P}  \int_X (\Lambda_{b, W}(n) - 1)\Psix(Wn + b) dx \right| = o_{W \rightarrow \infty}(1).
	\]
	As such for $W$ sufficiently large, by the triangle inequality
	\[
	\liminf_{P \in I'} \left| \E_{P/2 < Wn \leq P} \int_X \Psix(Wn + b) dx \right| \gtrsim_{k, s, c} 1.
	\]
	So far we exploited cancellation in the $\Lambda_{b, W}(n) - 1$ term and simply boundedness in the $\Psix(n)$ term. Next, we will try to exploit cancellation in $\Psix$ to obtain a contradiction. To exploit this cancellation we interpret the average as an integral over a complicated nilmanifold, then use the fact that the frequency of $\Phi$ is nontrivial on the identity component of $G / \Gamma$ and therefore nontrivial on every component of $G / \Gamma$. Let 
	$ G^{2t} = G \times \cdots \times G $ be the product of $G$ with itself $2t$ many times and let ${\bf g} = (g^{Wj_1}, g^{Wj_2},\ldots, g^{Wj_{2t}})$ be the element of $G^{2t}$ whose $\ell^{th}$ coordinate is $g^{Wj_\ell}$. For any $\sigma$ in $G$ let $\Delta \sigma = (\sigma, \ldots, \sigma)$ be the element of $G^{2t}$ whose entries are all $\sigma$ and let $G_\Delta$ be the set of all the elements of the form $\Delta \sigma$. Define
	\[
	{\bf G} = \overline{\langle {\bf g}, G_\Delta, \Gamma \times \cdots \times \Gamma \rangle},
	\]
	the closure of the group generated by ${\bf g}$, $G_\Delta$ and $\Gamma^{2t}$ inside $G^{2t}$. Our sequence $\Psix$ is a nilsequence on ${\bf G} / \Gamma^{2t}$. Consider the sequence of ``empirical" measures on $G^{2t} / \Gamma^{2t}$,
	\[
	\rho_P =  \E_{P/2 < Wn \leq P} (g^{(Wn + b)j_1}, \ldots, g^{(Wn+b)j_{2t}} )_* (\Delta_* dx), 
	\]
	where $\Delta_* dx$ is the Haar measure on $G_\Delta / (\Gamma^{2s} \cap G_\Delta)$ and where $*$ denotes the pushforward.
	By construction, if $\Xi \colon G^{2t} / \Gamma^{2t} \rightarrow \C$ is defined by 
	\[
	\Xi(x_1, \ldots, x_{2t}) = \prod_{j=1}^{t} \Phi(x_j) \cdot \prod_{j=t+1}^{2t} \overline{\Phi}(x_j),
	\] 
	then
	\[
	\E_{P/2 < Wn \leq P} \int_X \Psix(Wn + b) dx = \int_{G^{2t}/ \Gamma^{2t} } \Xi \ \ \rho_P(dx).
	\]
	By the Banach-Alaoglu theorem, there is a further subsequence along which the empirical measures converge weakly,
	\[
	\lim_{P \in I''} \rho_P \overset{*}{\rightharpoonup} \rho,
	\]
	where $I''$ is an infinite subset of $I'$. Note that, by summation by parts, $\rho_P$ is almost invariant by ${\bf g}$ in the following sense:
	\[
	{ \bf g }_* \rho_P = \rho_P + O\left(\frac{W}{P}\right)
	\]
	Therefore $\rho$ is actually ${\bf g}$ invariant.  Since $\rho$ is an average of $G_\Delta$ invariant measures, $\rho$ is a also $G_\Delta$ invariant. Of course $\rho$ is also $\Gamma^{2t}$ invariant because $\Gamma^{2t}$ acts trivially on $G^{2t} / \Gamma^{2t}$. Since stabilizers of measures are closed, $\rho$ is invariant under ${\bf G}$.
	By the classification of invariant measures, we know that $\rho$ is actually (a translate of) Haar measure on some nilmanifold ${\bf X}$. 
	Next we need the following result essentially due to Frantzikinakis \cite{Frantzikinakis}.
	\begin{lem}[\cite{Frantzikinakis}; see section 5.7 and especially the proof of Proposition 5.7]\label{frz}
		With all the notation as before, for any $u \in G_s$ and $m \leq s$, we have $( u^{(Wj_\ell)^m} )_{\ell=1}^{2t} \in {\bf G}$.
	\end{lem}
	We include the proof for completeness and because our result differs very slightly from the way it was stated in \cite{Frantzikinakis}.

\begin{proof}
We split Lemma \ref{frz} into three claims:
\begin{claim}\label{claim1}
Let $m$ and $\ell$ be natural numbers.
If $g_1$ is in  ${\bf G}_m$ and $g_2$ and $g_3$ are in  ${\bf G}_r$ then there exists $\sigma$ in $ {\bf G}_{m + r + 1}$ such that
\[
[g_1, g_2] \cdot [g_1, g_3] = [g_1, g_2 \cdot g_3] \cdot \sigma.
\]
Moreover, $\sigma$ depends continuously on $g_1$, $g_2$ and $g_3$. In fact, this holds for any nilpotent Lie group, not just  ${\bf G}$.
\end{claim}
\begin{claim}\label{claim2}
For any $m$ between $1$ and $s$ and any element $\tau$ in $ G_m$ and in the identity component (which is automatic for $m > 1$), there exists an element $\sigma$ in $ (G_{m+1})^{2t}$ such that 
\[
(\tau^{(Wj_\ell)^m } )_{\ell = 1}^{2t} \cdot \sigma \in {\bf G}.
\]
\end{claim}
\begin{claim}\label{frz2}
For any natural number $r$ between $1$ and $s$  and any natural number $m$ between $1$ and $r$ and for any $\tau$ in $G_r$ there exists an element $\sigma$ in $(G_{r+1})^{2s}$ such that
\[
(\tau^{(Wj_\ell)^m } )_{\ell = 1}^{2t} \cdot \sigma \in {\bf G}.
\]
\end{claim}
We remark that taking $\tau = u$ in Claim \ref{frz2} gives Lemma \ref{frz}.
\begin{proof}[Proof of Claim \ref{claim1}]
The proof is simply a computation. For any $g_1$, $g_2$ and $g_3$ as above
\begin{align*}
[g_1, g_2] \cdot [g_1, g_3] = & g_1 g_2 g_1^{-1} g_2^{-1} [g_1, g_3] \\
= & g_1 g_2 g_1^{-1} [g_1, g_3] g_2^{-1} \mod G_{m + r + 1} \\
= & g_1 g_2 g_3 g_1^{-1} g_3^{-1} g_2^{-1} \mod G_{m + r + 1} \\
= & [g_1, g_2 g_3] \mod G_{m + r + 1}
\end{align*}
\end{proof}
\begin{proof}[Proof of Claim \ref{claim2}]
We prove Claim \ref{claim2} by induction on $m$. First, suppose $m =1$. Consider the torus $Z = G / G_2 \Gamma$. Let $\pi$ be the projection map $\pi \colon G \rightarrow Z$. Then since $g$ acts ergodically on $G / \Gamma$, we know $\pi(g)$ is an ergodic element in $Z$. Therefore, for any $\pi(\tau)$ in $G/G_2$, note that $\pi(\tau)$ is in the orbit of $\pi(g)$. By the definition of ${\bf G}$, $(\pi(g^{Wj_\ell}))_{\ell = 1}^{2t}$ is an element of $\pi^{2t} ( {\bf G} )$.  Thus, for any $\tau$ in $G$, $(\pi(\tau^{Wj_\ell}))_{\ell = 1}^{2t}$ is an element of $\pi^{2t} ( {\bf G} )$ so by definition of the quotient
\[
(\tau^{Wj_\ell } )_{\ell = 1}^{2t} \cdot \sigma \in {\bf G}
\]
for some $\sigma$ in $G_2^{2t}$.

Next, assume by induction that Claim \ref{claim2} holds for $m$. We will try to prove the claim for $m+1$. We begin with the case where $\tau$ is the  commutator of two elements of the following form. Suppose that there exists $g_1$ in $G$ and $g_2$ in $G_m$ such that $[g_1, g_2] = \tau$. By assumption, there exists $\sigma_1$ in $G_2^{2t}$ and $\sigma_2$ in $G_{m+1}^{2t}$ such that
\[
(g_1^{Wj_\ell } )_{\ell = 1}^{2t} \cdot \sigma_1 \in {\bf G} \text{ and } (g_2^{(Wj_\ell)^m } )_{\ell = 1}^{2t} \cdot \sigma_2 \in {\bf G}.
\]
Since ${\bf G}$ is a group, we conclude that the commutator is in ${\bf G}$.
\[
[ (g_1^{Wj_\ell } )_{\ell = 1}^{2t} \cdot \sigma_1, (g_2^{(Wj_\ell)^m } )_{\ell = 1}^{2t} \cdot \sigma_2] \in {\bf G}.
\]
Using Claim \ref{claim1} repeatedly, this is
\[
([ g_1, g_2 ]^{W j_\ell^{m+1}} )_{\ell = 1}^{2t} \sigma \in {\bf G},
\]
for some $\sigma$ in $G_{m+2}^{2t}$.

Finally, we note that commutators generate $G_{m+1}$ so it suffices to show that if $\tau_1$ and $\tau_2$ are elements of $G_{m+1}$ that satisfy Claim \ref{claim2} then so does their product. After all, if
\[
(\tau_1^{(Wj_\ell)^{m+1} } )_{\ell = 1}^{2t} \cdot \sigma_1 \in {\bf G} \text{ and } (\tau_2^{(Wj_\ell)^{m+1} } )_{\ell = 1}^{2t} \cdot \sigma_2 \in {\bf G},
\]
where $\sigma_1$ and $\sigma_2$ are in $G_{m+2}^{2t}$ then 
\begin{align*}
& (\tau_1^{(Wj_\ell)^{m+1} } )_{\ell = 1}^{2t} \cdot \sigma_1 \cdot (\tau_2^{(Wj_\ell)^{m+1} } )_{\ell = 1}^{2t} \cdot \sigma_2  \\
= & (\tau_1^{(Wj_\ell)^{m+1} } )_{\ell = 1}^{2t} \cdot (\tau_2^{(Wj_\ell)^{m+1} } )_{\ell = 1}^{2t}  \cdot \sigma  \\
= & ((\tau_1 \tau_2)^{(Wj_\ell)^{m+1} } )_{\ell = 1}^{2t} \cdot \sigma'  \\
= & (\tau^{(Wj_\ell)^{m+1} } )_{\ell = 1}^{2t} \cdot \sigma',
\end{align*}
where $\sigma$ and $\sigma'$ are in $G_{m+2}^{2t}$. This completes the proof of Claim \ref{claim2}.
\end{proof}

\begin{proof}[Proof of Claim \ref{frz2}]
First, if $m = r$ then we are done by Claim \ref{claim2}. Thus, we will assume $m < r$.

Second, we check that if $\tau_1$ and $\tau_2$ are in $G_r$ and satisfy Claim $\ref{frz2}$ then so does their product. By assumption, we may write
\[
(\tau_i^{(Wj_\ell)^m } )_{\ell = 1}^{2t} \cdot \sigma_i \in {\bf G},
\]
where $\sigma_i$ is an element of $(G_{r+1})^{2s}$ and $i = 1, 2$. Then the product is given by
\begin{align*}
&(\tau_1^{(Wj_\ell)^m } )_{\ell = 1}^{2t} \cdot \sigma_1 \cdot (\tau_2^{(Wj_\ell)^m } )_{\ell = 1}^{2t} \cdot \sigma_2 \\ = & (\tau_1^{(Wj_\ell)^m } )_{\ell = 1}^{2t} \cdot (\tau_2^{(Wj_\ell)^m } )_{\ell = 1}^{2t} \cdot \sigma_1 \cdot  [\sigma_1^{-1}, (\tau_2^{-(Wj_\ell)^m } )_{\ell = 1}^{2t} ] \cdot \sigma_2.
\intertext{
Then we use that $\tau_1^{(Wj_\ell)^m} \cdot \tau_2^{(Wj_\ell)^m} = (\tau_1 \tau_2)^{(Wj_\ell)^m}$ up to higher order terms.
}
= & ((\tau_1 \tau_2)^{(Wj_\ell)^m } )_{\ell = 1}^{2t} \mod (G_{r+1})^{2s}.
\end{align*}
Therefore, it suffices to prove Claim \ref{frz2} in the case that $\tau = [g_1, g_2]$ where $g_1$ is in $G_m$ and $g_2$ is in $G_{r-m}$ because such commutators generate $G_r$ as a group up to higher order corrections. 

By Claim \ref{claim2}, there exists $\sigma$ in $(G_{m+1})^{2s}$ such that
\[
(g_1^{(Wj_\ell)^m } )_{\ell = 1}^{2t} \cdot \sigma \in {\bf G}.
\]
We also know, because ${\bf G}$ contains diagonal elements, that $(g_2)_{\ell = 1}^{2t}$ is an element of ${\bf G}$. We conclude that 
\[
[(g_1^{(Wj_\ell)^m } )_{\ell = 1}^{2t} \cdot \sigma, (g_2)_{\ell = 1}^{2t}] \in {\bf G}.
\]
By Claim \ref{claim1}, this is given by 
\[
(\tau^{(Wj_\ell)^m } )_{\ell = 1}^{2t} \cdot \sigma' \in {\bf G},
\]
for some $\sigma'$ in $G_{r+1}$. 
\end{proof}
This completes the proof of Lemma \ref{frz} by plugging in $r = s$.
\end{proof}

	Since the frequency $\xi$ of $\F$ is nontrivial on the identity component, there exists an element $u$ in the identity component of $G_s$ such that $\frac{1}{2\pi i}\log \xi(u)$ is irrational. Fix such a $u$. Now since $J$ does not solve Vinogradov's mean value problem there exists $m \leq s$ such that $j_1^m + \cdots + j_t^m - j_{t+1}^m - \cdots - j_{2t}^m \neq 0$. Fix such an $m$. Then the map $G_s \rightarrow G_s$ given by $v \mapsto v^{(j_1^m + \cdots + j_t^m - j_{t+1}^m - \cdots - j_{2t}^m) W^m}$ has image both open and closed so $u$ is in the image. For more details, see \cite{Frantzikinakis}. Fix a $v$ such that $v \mapsto u$. Then by Lemma \ref{frz}, $(v^{(Wj_\ell)^m})_{\ell =1}^{2t} \in {\bf G}$. As such
	\begin{align*}
	\int_{\bf X} \Xi(x) \ \rho(dx) =& \int_{\bf X} \Xi(v x) \ \rho(dx) \\
	=& \int_{\bf X} \xi(u) \Xi(x) \ \rho(dx) \\
	=& 0.
	\end{align*}
	This gives a contradiction. We conclude that the terms which do not solve Vinogradov's mean value problem do not contribute to our sum.

	For every $2t$-tuple $j_1, \ldots, j_{2t}$ in $[k]^{2t}$, we have that for all $p$
\[
\big| \F(g^{pj_1} x) \cdot \cdots \cdot  \F(g^{pj_t} x) \cdot \overline{\F}(g^{pj_{t+1}} x) \cdot \cdots \cdot \overline{\F}(g^{pj_{2t}} x) \big| \leq 1,
\]
simply using a trivial $L^\infty$ bound.
	For every $2t$-tuple which does not solve Vinogradov's mean value problem we have 
\[
\limsup_{k \in \K} \limsup_{P \rightarrow \infty} \E_{P/2 < p \leq P} \int_X \F(g^{pj_1} x) \cdot \cdots \cdot  \F(g^{pj_t} x) \cdot \overline{\F}(g^{pj_{t+1}} x) \cdot \cdots \cdot \overline{\F}(g^{pj_{2t}} x)= 0.
\]
	Therefore, the average is bounded by the fraction of terms which solve Vinogradov's mean value problem. There are no more than $C_{s, \e} k^{t + .5 \e}$ such solutions by Bourgain-Demeter-Guth (Theorem \ref{BDG}). Thus
	\[
	\limsup_{k \in \K} \limsup_{P \rightarrow \infty} \E_{P/2 < p \leq P} \int_X | \E_{h \leq k} \epsilon_h \Phi(g^{ph} x) |^{2t} dx \cdot k^{t - \e} = 0,
	\]
	in the case $t \neq 2$ and 
	\[
	\limsup_{k \in \K} \limsup_{P \rightarrow \infty} \E_{P/2 < p \leq P} \int_X | \E_{h \leq k} \epsilon_h \Phi(g^{ph} x) |^{2t} dx \cdot k^{t}  \leq C,
	\]
	in the case $t = 2$. After all, since diagonal solutions are the only solutions to Vinogradov's mean value problem in the case of two variables and one equation i.e. $j_1 = j_2$, there is no $\e$ loss when $t = 2$. 
	Thus, we obtain Theorem \ref{bigthm} and in turn Theorem \ref{mainthm} and Theorem \ref{condmainthm}.

	\section{Proof of Theorem \ref{approxmainthm}}\label{section3}
	The proof of Theorem \ref{approxmainthm} is essentially the proof of Theorem \ref{mainthm} with a few minor simplifications. As before suppose not. Then as before, we can find a joining such that
	\[
	\left| \int_{X \times Y} f(x) f'(y)  \nu(dx dy) \right| > c.
	\]
	As before, we can apply \cite{FrantzikinakisHost2} such that
	\[
	\left| \int_{X \times Y} \E^\mu[f | \z](x) f'(y)  \nu(dx dy) \right| > c.
	\]
	Unlike before, we do not need to restrict the integral to $B$.
	As before, we can average over translates
	\[
	\int_{X \times Y} | \E_{h \leq k} \E^\mu[f | \z](T^h x) f'(T^h y) |  \nu(dx dy) > c.
	\]
	As before, we can take an essential supremum over $y$
	\[
	\int_{X} \sup_{y \in Y} | \E_{h \leq k} \E^\mu[f | \z](T^h x) f'(T^h y) | \mu(dx) > c.
	\]
	As before, we can apply the entropy decrement argument, for some $P \gg k$, we have
	\[
	\E_{P/2 < p \leq P} \int_{X} \sup_{y \in Y} | \E_{h \leq k} \E^\mu[f | \z](T^{ph} x) f'(T^h y) | \mu(dx) > c.
	\]
	We can use the Cauchy-Schwarz inequality
	\[
	\E_{P/2 < p \leq P} \int_{X} \sup_{y \in Y} | \E_{h \leq k} \E^\mu[f | \z](T^{ph} x) f'(T^{h} y) |^2 \mu(dx) > c^2.
	\]
	This time, would like to replace $f'$ by a sum over words of length $k$ up to $\e$ rounding. In the no-rounding case, we knew that words of $f'$ were words of $b$. We double check that a similar result holds for words up to constant rounding. In particular, fix $k$ such that there are at most $\delta k$ words of length $k$ that occur with positive $\log$ density up to $\e$ rounding. Thus, we can fix a set $\Sigma$ of words of length $k$ such that $\# \Sigma \leq \delta k$ and for all $n$ outside a set of $0$ $\log$ density there exists an $\epsilon$ in $\Sigma$ such that $| b(n +h) - \epsilon_h | \leq \e$. Translating this to the dynamical setting,
	\[\nu \{ (x,y) \colon \text{there exists $\epsilon$ in $\Sigma$ such that } |f'(T^h y) - \epsilon_h| \leq \e \} = 0. \]
	Therefore, we can replace $f'$ by a sum over words as before.
	\[
	\sum_{\epsilon} \E_{P/2 < p \leq P} \int_{X} | \E_{h \leq k} \E^\mu[f | \z](T^{ph} x) \epsilon_h |^2 \mu(dx) > c^2 - 2\e.
	\]
	Notice that this time, when we replace $f'(T^h y)$ by a word, we incur an error of $\e$.
	Now the rest of the argument runs exactly the same as before. In fact, after pigeonholing, any dependence on $b$ completely drops out of the argument.
	\appendix
	\section{Frantzikinakis-Host and dynamical models}
	\cite{TaoBlog} shows that there is a joining $(X_0 \times Y, \nu_0, T, f, f', M, I_m)$ of a dynamical model for $a$ with $b$ where $X_0 = D^\Z \times \widehat{\Z}$ is the space of sequences in the unit disk with the product topology, $T$ is the shift map on $D^\Z$ and $+1$ on $\widehat{\Z}$, $f$ is the evaluation at $0$ map, $M$ is projection onto the second factor and $I_m( (x(n))_{n \in \Z}, r  ) = ( (\overline{a(m)} x(m n))_{n \in \Z}, \frac{r}{m}  )$ whenever $r$ is in $m\widehat{\Z}$.
	Call $\mu_0$ the pushforward of $\nu_0$ onto $X_0$. 
	Of course, $X_0$ factors onto $D^\Z$ by projection onto the first factor. Call $\rho$ the pushforward of $\nu$ onto $D^\Z$.
	\cite{FrantzikinakisHost2} (Proposition 4.2 in that paper) showed that $(D^\Z, T, \rho)$ is a factor of a system $(\tilde{X}, \tilde{\rho}, T)$ where $\tilde{X} = (D^\Z)^\Z$, $T$ is the shift map and there exists a natural number $d$ so that if $\P_d$ is the set of primes which are $1 \mod d$ then 
	\[
	\int_{\tilde{X}} \prod_{j = -K}^K F_j (T^j x) \tilde\rho(dx) = \lim_{N \rightarrow \infty} \E_{p \in \P_d \cap [N]} \int_{D^\Z} \prod_{j=-K}^K F_j(T^{pj} x) \rho(dx),  
	\]
	where $K$ is any natural number, the functions $F_j$ are any bounded measurable functions depending only on the $0^{th}$ coordinate and by \cite{FrantzikinakisHost2} the limit always exists. We fix such a $d$. 
	By \cite{FrantzikinakisHost2} (see Theorem 4.5 in that paper), each ergodic component of $\tilde{X}$ is isomorphic to a product of a Bernoulli system with an inverse limit of nilsystems. Thus, we get a joining of $\tilde{X}$ with $X_0$ over their common factor $D^\Z$. Call this joining $(X, \mu, T)$. We also get a joining of $X_0 \times Y$ and $X$ over their common factor $X_0$, which we call $(X \times Y, \nu, T)$. 
Explicitly, this joining is defined as follows. A point in $(X \times Y, \nu, T)$ can be thought of as a triple of points $(x_1, x_2, y)$ with $x_1$ in $\tilde X$ and $(x_2, y)$ in $X_0 \times Y$. Since $X_0 = D^\Z \times \widehat \Z$, we have that $x_2 = (x_3, r)$ for some $x_3$ in $D^\Z$ and $r$ in $\widehat \Z$. The measure is supported on triples where $\pi(x_1) = x_3$ so we will often forget $x_1$ and simply write a point in $X \times Y$ as a triple $(x, r, y)$ with $x$ in $\tilde X$, $r$ in $\widehat \Z$ and $y$ in $Y$. The measure is given explicitly by the following formula: if $K$ is a natural number, $F_j$ are bounded measurable functions on $D^\Z$ depending only on the $0^{th}$ coordinate, $\phi$ is a bounded measurable function on $\widehat \Z$ and $\psi$ is a bounded measurable function on $Y$ then
\begin{align*}
&\int_{X \times Y} \prod_{j = -K}^{K} F_j(T^j x) \cdot \phi(r) \cdot \psi(y) \nu(dx dr dy) \\ =  \lim_{N \rightarrow \infty} \E_{p \in \P_d \cap [N]} &\int_{X_0 \times Y} \prod_{j = -K}^{K} F_j(T^{pj} x) \cdot \phi(r) \cdot \psi(y) \nu_0(dx dr dy).
\end{align*}

We will proceed to check that $X \times Y$ has all the desired properties. We define $M \colon X \rightarrow \widehat{\Z}$ by taking an element $(x, r)$ with $x$ in $\tilde{X}$ and $r$ in $\widehat{\Z}$ to $r$. Let $x$ be an element of $\tilde{X}$. We will write $x = (x_n)_{n \in \Z}$ for a sequence of elements $x_n$ in $D^\Z$ and write $x_n(k) \in D$ for the $k^{th}$ element of the sequence $x_n$. Let $\iota_m((x(k))_{k \in \Z}) = \overline{a(m)} (x(mk))_{k \in \Z} $ We define $I_m(x,r) = (\iota_m( (x_{nm})_{n \in \Z} ), \frac{r}{m})$. Explicitly
	\[
	I_m( x,  r) = \left( (\overline{a(m)} x_{nm}(km)_{k \in \Z}  )_{n \in \Z}, \frac{r}{m}  \right)
	\]
	whenever $r$ is in $m\widehat{\Z}$. We define $f \colon X \times Y \rightarrow \C$ by the formula $f(x,r,y) = x_0(0)$. This is just the pullback of $f \colon X_0 \times Y \rightarrow \C$ under the factor map $X \times Y \rightarrow X_0 \times Y$. We define $f' \colon X \times Y \rightarrow \C$ by pulling back  $f' \colon X_0 \times Y \rightarrow \C$ under the same factor map i.e. $f'(x, r, y) = y$.
Now we check
	\begin{itemize}
		\item $M(T(x, r)) = M(Tx,r+1) =  r+1 = M(x, r) +1$
		\item We have \begin{align*}
		I_m \circ T^m(x,r) =& \left( (\overline{a(m)} x_{nm+m}(km)_{k \in \Z}  )_{n \in \Z}, \frac{r+m}{m}   \right) \\
		=& \left( ( \overline{a(m)} x_{(n+1)m}(km)_{k \in \Z}  )_{n \in \Z}, \frac{r}{m} +1  \right) \\
		=& T\left( ( \overline{a(m)} x_{nm}(km)_{k \in \Z}  )_{n \in \Z}, \frac{r}{m}  \right) \\
		=& T \circ I_m (x, r).
		\end{align*}
		for any $m$ and whenever $r$ is in $m\widehat{\Z}$.
		\item Let $K$ be a natural number and $F_j \colon \tilde{X} \rightarrow \C$ be a sequence of bounded measurable functions depending only on $0$. Let $\phi$ be a function which is measurable with respect to $\widehat{\Z}$. Then for any $m$,
		\begin{align*}
		&\int_{X} \ch_{M^{-1}(m\widehat{\Z})}(x) \phi(I_m x) \cdot \prod_{j = -K}^K F_j(T^j I_m x) \mu(dx) 
		\\=& \lim_{N \rightarrow \infty} \E_{p \in \P_d \cap [N]} \int_{X_0} \ch_{r \in m \widehat{\Z}} \phi\left( \frac{r}{m} \right) \cdot \prod_{j = -K}^K F_j \circ \iota_m ( T^{pmj} x) \mu_0(dxdr),
		\intertext{by definition of $\mu$. Next, we use that $\iota_m \circ T^{pmj} = T^{pj} \circ \iota_m$.}
		=& \lim_{N \rightarrow \infty} \E_{p \in \P_d \cap [N]} \int_{X_0} \ch_{r \in m \widehat{\Z}} \phi\left( \frac{r}{m} \right) \cdot \prod_{j = -K}^K F_j \circ T^{pj} \circ \iota_m ( x) \mu_0(dxdr).
		\intertext{Because $X_0$ is a dynamical model for $a$, it possesses a dilation symmetry,}
		=& \lim_{N \rightarrow \infty} \E_{p \in \P_d \cap [N]} \int_{X_0} \frac{1}{m} \phi(r) \cdot \prod_{j = -K}^K F_j( T^{pj} x  ) \mu_0(dxdr).
		\intertext{Finally, we apply the definition of $\mu$ one more time,}
		=&\int_{X} \frac{1}{m} \phi(x) \cdot \prod_{j = -K}^K F_j(T^j  x) \mu(dx). 
		\end{align*}
		\item For any natural number $m$ and any $r$ in $m \widehat{\Z}$, we have
		\[
		f (I_m (x,r) ) = f( \left( (\overline{a(m)} x_{nm}(km)_{k \in \Z}  )_{n \in \Z}, \frac{r}{m}  \right) ) = \overline{a(m)} x_0(0) = \overline{a(m)} f(x,r).
		\]
		\item Clearly, for any natural numbers $m$ and $h$,
		\[
		I_h I_m( x,  r) = \left( (\overline{a(mh)} x_{nmh}(kmh)_{k \in \Z}  )_{n \in \Z}, \frac{r}{mh}  \right) = I_m I_h (x, r),
		\]
		for any $r$ in $hm \widehat{\Z}$.
		\item Since $f$ and $f'$ are pulled back from $X_0 \times Y$, the ``statistics" of $f \colon X \times Y \rightarrow \C$ will be the same as the statistics of $f \colon X_0 \times Y \rightarrow \C$ and similarly for $f'$. 
	\end{itemize}
	Therefore, $(X \times Y, \nu, T, f, f', I_m, M)$ is a joining of a dynamical model for $a$ with $b$. 
	
	Let $(X, \mu_\omega, T)$ be an ergodic component of $(X, \mu, T)$ which joins the corresponding ergodic component $(\tilde X, \tilde\rho_\omega, T)$ of $(\tilde X, \tilde\rho, T)$ with $\widehat{\Z}$. Note that $\widehat{\Z}$ is already an ergodic inverse limit of nilsystems: after all it is an inverse limit of the ergodic systems of the form $\Z / m \Z$ and the inverse limit of ergodic systems is ergodic. By \cite{FrantzikinakisHost2}, there is a Bernoulli system $(W, dw, T)$ and an inverse limit of nilsystems $(Z_0, dz, T)$ such that $(\tilde X, \tilde\rho_\omega, T) \cong (W, dw, T) \times (Z_0, dz, T)$. 
Therefore $(X, \mu_\omega, T)$ is isomorphic to $(W \times Z_0 \times \widehat{\Z}, \mu', T)$ where $\mu'$ is some mystery measure and where $T$ is just the product transformation. We can think of this system as a joining of $(W \times Z_0, dw \times dz, T)$ with $(\widehat{\Z}, dz, T)$ or we can think of this system as a joining of $(W, dw, T)$ with $(Z_0 \times \widehat{\Z}, \zeta, T)$ where $\zeta$ is some unknown measure given by pushing forward $\mu'$ onto $Z_0 \times \widehat{\Z}$. 
Next, we claim that any ergodic joining of two inverse limits of nilsystems is in fact isomorphic to an inverse limit of nilsystems. After all, if $Z_1$ and $Z_2$ are two nilsystems and $\zeta$ is an ergodic invariant measure on $Z_1 \times Z_2$, then $\zeta$ is a translate of Haar measure on some closed subgroup by measure classification for nilsystems. Thus $(Z_1 \times Z_2, \zeta, T) \cong (Z_3, dz, T)$ for some nilsystem $Z_3$. Taking inverse limits, $(Z_0 \times \widehat{\Z}, \zeta, T)$ is isomorphic to an inverse limit of nilsystems $(Z, dz, T)$. Because $(Z, dz, T)$ is an inverse limit of nilsystems, it has zero entropy so the only possible joining of $(Z, dz, T)$ with the Bernoulli system $(W, dw, T)$ is the trivial joining i.e. $\mu'$ is the product measure $dw \times dz$.
 Lastly, we claim that $Z$ is isomorphic to the Host Kra factor of $(X, \mu_\omega, T)$. Since the Host Kra factor $\z(X)$ is isomorphic to an inverse limit of nilsystems, it has zero entropy, so any factor map from $W \times Z$ to $\z(X)$ where $W$ is Bernoulli necessarily factors through $Z$. Thus $Z$ factors onto $\z(X)$. Of course, since $X$ factors onto $Z$, the Host Kra factor for $X$ factors onto the Host Kra factor for $Z$. Implicitly in \cite{HostKra} and explicitly, for instance, in \cite{HostKra2} chapter 12, for any nilsystem $(Z_1, dz, T)$ the Host Kra factor of $Z_1$ is $Z_1$. Thus, taking inverse limits gives that the Host Kra factor of $Z$ is $Z$ so $\z(X) \cong Z$. This completes the proof.
 
\section{Reduction to the completely multiplicative case}

We have stated our main theorems in the case that $a$ is completely multiplicative. In this appendix, we show that these assumptions can be weakened to include all multiplicative functions. For example, we will show that Theorem \ref{mainthm} holds in this generality. The same argument works for Theorem \ref{condmainthm} and Theorem \ref{approxmainthm} (although in this last case, the way that $c$ depends on $\e$ gets worse). 
The argument here will be entirely formal, using nothing of the proof of Theorem \ref{mainthm} and only the result. However, we remark that the interested reader could check that the proof we give can be adapted to the more general case of multiplicative functions. The main difference is that now the dynamical model for $a$ does not satisfy the identity that the push forward of $\mu$ restricted to $M^{-1}(m\widehat{\Z})$ is $\frac{1}{m} \mu$ but instead we incur a $\frac{1}{m}$ error i.e. for all $\phi$ in satisfying $|| \phi ||_{L^\infty(\mu)} \leq 1$ we have
\[
\int_X \phi(x) \mu(dx) = \int_X m \ch_{x \in M^{-1} (m \widehat{\Z} )} \phi(I_m(x)) \mu(dx) + O\left(\frac{1}{m}\right).
\]
This introduces an error term of size $O\left( \frac{1}{P} \right)$ in Corollary \ref{cor1} which tends to $0$ as $P$ tends to infinity.

However, here we proceed just using the statement of Theorem \ref{mainthm}.
Famously, we can write
\[
\mu(n) = \sum_{d^2 | n} \lambda\left( \frac{n}{d^2} \right) \mu(d),
\]
where $\lambda$ is Liouville function and $\mu$ is the M\"obius function which agree with the Liouville function on squarefree numbers and vanishes on numbers which are not squarefree. Of course,
\[
\lambda\left( \frac{n}{d^2} \right) = \lambda(n),
\]
but we write it this way to suggest that the convolution identity
\[
\mu = \lambda * \phi
\]
where $\phi(d^2) = \mu(d)$ may be generalized. In fact, for any multiplicative function $a$ taking values on the unit circle, we may write
\[
a = a_1 * a_2
\]
where $a_1$ is some completely multiplicative function taking values on the unit circle and $a_2$ is a (possibly unbounded) multiplicative function supported on numbers of the form $d^k$ for some natural numbers $d$ and $k$ with $k \geq 2$. To prove this is possible, it suffices to check it is possible on prime powers since both sides are multiplicative. For any prime $p$, we define 
\[
a_1(p) = a(p)
\]
and so 
\begin{align*}
a_1(p) \cdot a_2(1) + a_1(1) \cdot a_2(p) &= a(p) \cdot 1 + 1 \cdot 0 = a(p).
\end{align*}
We also want,
\begin{align*}
a(p^2) = & a_1(p^2) \cdot 1 + 1 \cdot a_2(p^2)
\intertext{so we choose}
a_2(p^2) = & a(p^2) - a_1(p^2) \\
= & a(p^2) - a(p)^2.
\intertext{
	Iteratively, we may define
}
a_2(p^k) =& a(p^k) - \sum_{0 \leq i < k} a_1(p^{k-i}) a_2 (p^i) \\
=& a(p^k) - \sum_{0 \leq i < k} a(p)^{k-i} a_2 (p^i).
\end{align*}
Since whether $a$ is unpretentious or not depends only on the behavior of $a$ at primes, clearly if $a$ is unpretentious then so is $a_1$.

Informally, the probability that a random number is divisible by $d$ is roughly $\frac{1}{d}$. Thus, the expected number of times that any number of the form $d^k$ for $k \geq 2$ divides a random natural number is at most
\[
\sum_{d \geq 2} \sum_{k \geq 2} \frac{1}{d^k}
\]
which is summable. Thus the tails 
\[
\sum_{d \geq C} \sum_{k \geq 2} \frac{1}{d^k}
\]
and
\[
\sum_{d \geq 2} \sum_{k \geq C} \frac{1}{d^k}
\]
tend to zero as $C$ tends to infinity. Let $S$ be the set of natural numbers $n$ for which $d^k$ divides $n$ for $d, 
k \geq 2$ implies $d, k \leq C$. The previous analysis says most numbers are in $S$. Fix a function $b$ as in the statement of Theorem \ref{mainthm}, that is a bounded function such that for any $\delta > 0$ there are infinitely many $k$ such that the number of words of $b$ of length $k$ that occur with positive upper logarithmic density is at most $\delta k^2$. Our goal will be to show that for $N$ large,
\[
| \E_{n \leq N}^{\log} a(n) b(n) |
\]
is small, say less than a constant times some small positive number $\e$. If $C$ is sufficiently large depending on $\e$ but still very small compared to $N$, we may modify $a$ on the set of numbers outside $S$. In particular, $a$ is given by the formula
\begin{align*}
a(n) = &\sum_{\ell | n} a_1\left( \frac{n}{\ell} \right) a_2( \ell )
\intertext{
	For most numbers, this is the same as
}
= &\sum_{ \substack{\ell | n \\ \ell \leq C^C } } a_1\left( \frac{n}{\ell} \right) a_2( \ell ).
\intertext{
	That formula works as long as $n$ is not divisible by a number of the form $d^k$ where either $d$ or $k$ is greater than $C$. In that exceptional case when $n$ is divisble by a number of the form $d^k$ with $d$ or $k$ greater than $C$, we can write $n$ as $i \cdot j$ where $i$ is not divisible by any number greater than $C^C$ and is as large as possible given that constraint. We conclude that if $\ell \leq C^C$ and $\ell | n$ then $\ell | i$. Thus, expanding the definitions and using multiplicativity,
}
a_1(j) \cdot a(i) = & a_1(j) \sum_{ \substack{\ell | i \\ \ell \leq C^C } } a_1\left( \frac{i}{\ell} \right) a_2( \ell ) \\ 
=&\sum_{ \substack{\ell | i \\ \ell \leq C^C } } a_1\left( \frac{i j}{\ell} \right) a_2( \ell ) \\
=&\sum_{ \substack{\ell | n \\ \ell \leq C^C } } a_1\left( \frac{n}{\ell} \right) a_2( \ell ) \\
\end{align*}
We conclude that the formula
\[
\sum_{\ell | n} a_1\left( \frac{n}{\ell} \right) a_2( \ell )
\]
is bounded and agrees with $a(n)$ all but at most
\[
\sum_{d \geq C} \sum_{k \geq 2} \frac{1}{d^k} +\sum_{d \geq 2} \sum_{k \geq C} \frac{1}{d^k}
\]
of the time. Thus, it suffices to show
\begin{align*}
& \left| \E_{n \leq N}^{\log} \sum_{\substack{m \ell = n \\ \ell \leq C^C}} a_1(m) a_2(\ell) b(n) \right|
\intertext{
	is small. By changing variables and applying Fubini,
}
= & \left| \sum_{\ell \leq C^C} a_2(\ell) \E_{m \leq N/\ell}^{\log}  a_1(m)  b(m \ell) \right|.
\end{align*}		
Fix a natural number $\ell$. Notice that every word of length $k$ of the function $m \mapsto b(m \ell)$ embeds in a word of $b$ of length $k \cdot \ell$. Thus, it is easy to check that $m \mapsto b(m \ell)$ still satisfies the conditions of Theorem \ref{mainthm}. Therefore, as $N$ tends to infinity, the previous expression tends to $0$.


\end{document}